\documentclass[11pt]{amsart} 
\usepackage{verbatim, latexsym, amssymb, amsmath,color,mathrsfs}
\usepackage{esint}
\usepackage{bm}
\usepackage{enumitem}
\usepackage[hidelinks]{hyperref}
\usepackage{xcolor}

\def\R{\mathbb R}
\def\Q{\mathbb Q}

\def\RP{\mathbb {RP}}
\def\N{\mathbb N}
\def\Z{\mathbb Z}

\newcommand{\T}{\mathbb T}

\def\tr{\operatorname{tr}}

\def\vol{\mathrm{vol}}

\def\area{\mathrm{area}}

\newtheorem{thm}{Theorem}[section]
\newtheorem{lemm}[thm]{Lemma}
\newtheorem{cor}[thm]{Corollary}
\newtheorem{prop}[thm]{Proposition}
\theoremstyle{remark}

\theoremstyle{definition}

\title[]{Rigidity and non-rigidity of the stable norm on $T^n$}
\author{Fernando C. Marques, Andr\'e Neves, and Ao Sun}
\address{Princeton University \\ Fine Hall \\ Princeton NJ 08544 \\ USA}
\email{coda@math.princeton.edu}
\address{University of Chicago \\ Department of Mathematics \\ Chicago IL 60637\\ USA}
\email{aneves@math.uchicago.edu}
\address{Lehigh University \\ Department of Mathematics \\ Bethlehem PA 18015\\ USA}
\email{{aosun@lehigh.edu}}
\thanks{The first author is partly supported by NSF-DMS-2506810 and a Simons Investigator Grant. The second author is partly supported by NSF-DMS-2005468 and a Simons Investigator Grant.}

\begin{document}
	\maketitle
	\begin{abstract}
		We show that the stable norm of flat metrics on $H_{d}(T^n,\R)$  is locally rigid if $1\leq d<n-1$ and locally rigid among metrics of the same volume if $d=n-1$. We also show that the stable norm on  $H_{2}(T^3,\R)$ is not locally rigid.
		
		As applications, we answer negatively a question raised by Bangert in his ICM address,  prove local rigidity of the marked $k$-area spectrum of flat metrics for $1\leq k\leq n-2$, and prove a local rigidity result for the volume spectrum of flat metrics on $T^n$.
	\end{abstract}
	\section{Introduction}
	A classical problem in geometry has been to determine whether the lengths of geodesics determine the ambient metric. Most positive results concern the ambient space being negatively curved. 
	
	The marked length spectrum assigns to each free homotopy class the length of its shortest representative.
	In the 1980s, Katok (see \cite{BurnsKatok}) conjectured that if two negatively curved metrics $g_1$ and $g_2$ have the same marked length spectrum, then they must be isometric. In dimension two, this was first proved by Croke \cite{Croke} and Otal \cite{Otal} independently. The local version of the analogous rigidity conjecture in higher dimensions was proved in \cite{GuillarmouLefeuvre} and Hamenst\"adt \cite{Hamenstaedt1999} showed that hyperbolic metrics are uniquely determined by their marked length spectrum.
	
	On $T^n=\R^n/\Z^n$ one can construct non-isometric metrics with the same marked length spectrum and so the natural question is to know whether Euclidean metrics are uniquely determined by their marked length spectrum. When $n=2$, Bangert \cite[Theorem 6.1]{bangert} showed that this is the case. For $n>2$ the question remains open. We present the most general known result.
	
	A useful object to understand this problem is the stable norm $||\cdot||_{s,g}$ on $H_d(T^n,\R)$ associated to a metric $g$ on $T^n$. Given $w\in H_d(T^n,\R)$, $||w||_{s,g}$ is defined to be the infimum of ${\bf M}_g(\Sigma)$ over all cycles that represent $w$ (see \eqref{stable.norm.defi}). If two metrics on $T^n$ have the same marked length spectrum, then they will have the same stable norm on $H_1(T^n,\R)$ \cite[Proposition~2.1]{Bangert1990}.
	
	Burago--Ivanov \cite{burago-ivanov-gafapaper} defined the asymptotic volume $AV(g)$ of a metric $g$ on $T^n$ and showed that it is uniquely minimized by flat metrics.  One consequence is that if $g$ has the same stable norm on $H_1(T^n,\R)$ as a flat metric $g_0$, then $\vol_g(T^n)\geq \vol_{g_0}(T^n)$ and equality implies $g$ is flat. A self-contained proof of the inequality is given in Section \ref{stable.norm.appendix} for the reader's convenience. The argument is well known.

	Our first result removes the requirement that  $\vol_g(T^n)=\vol_{g_0}(T^n)$ at the expense of having $g$ close in the smooth topology to $g_0$.
	We say that a $d$-plane $P\subset \R^n$ is rational if $P\cap \Z^n$ has rank $d$. In this case 
	$$
	T_P:=P/(P\cap \Z^n)
	$$
	is an embedded primitive $d$-dimensional subtorus of $T^n$. {Throughout this paper, for simplicity, we use ${\bf M}$ to denote the mass under the flat metric $g_0$.}
	\begin{thm}\label{higher.codimention.intro}
		Consider  $g_0$ a unit-volume flat metric on $T^n$ and $1\leq d\leq n-2$. There is $\mathcal U$ an  open neighborhood of $g_0$ in the smooth norm, so that if $g\in \mathcal U$ is such that for every rational $d$-plane $P$,
		$$
		{{\bf M}(T_P)=||T_P||_{s,g},}
		$$
		then $g$ is isometric to $g_0$.
	\end{thm}
	It is {still} an open problem to know whether the condition  $g\in \mathcal U$ can be removed. 
	
	A similar quantity to the stable norm appears in the context of convex billiards under the name of Mather's $\beta$-function (see \cite[Sections~1.4.1--1.4.2]{FKS24} or \cite{FLS25}, for instance). The local rigidity problem for ellipses was solved by Kaloshin--Sorrentino \cite[Corollary~13]{KS18}.

	We address now the codimension-one case. When $d=n-1$, Federer \cite{Federer74} showed that if $w\in H_{n-1}(T^n,\Z)$, then
	$$||w||_{s,g}=\inf\{{\bf M}_{g}(S):S\in \mathcal Z_{n-1}(T^n,\Z),\,[S]=w\},$$
	where $\mathcal Z_{n-1}(T^n,\Z)$ denotes the space of integral cycles. In the codimension-one case the stable norm carries more information regarding the behavior of minimal surfaces. 
	
	{Mather \cite{mather} and Bangert \cite{bangert} showed by different methods that  the unit stable sphere in $H_1(T^2,\R)$ is $C^1$ if and only if every primitive homology class in $H_1(T^2,\Z)$ admits a foliation by shortest closed geodesics. Later, Chambolle--Goldman--Novaga \cite{cgn} extended this result and showed that if the unit stable sphere in $H_{n-1}(T^n,\R)$ is $C^1$, then every primitive homology class in $H_{n-1}(T^n,\Z)$ admits a foliation by plane-like closed area-minimizing hypersurfaces (see also \cite{auer-bangert}). The existence of plane-like minimizers in every direction was proved by Caffarelli--de la Llave in \cite{caffarelli-lalave} (see also \cite{bangert-laminations}) and, for bumpy metrics, the existence of plane-like non-area-minimizers in totally irrational laminations was proven by Nguyen in \cite{Nguyen26}.  The study of periodic variational problems for a general elliptic functional was originally initiated by Moser \cite{moser} who noticed that many principles of Aubry--Mather theory could be applied in this setting.}

	There is no known rigidity result for the stable norm on $H_{n-1}(T^n,\Z)$; {even assuming the extra condition that  the stable norm on $H_{1}(T^n,\Z)$ is also preserved} \footnote{The extra hypothesis that $\vol_g(T^n)=\vol_{g_0}(T^n)$ is not stated in \cite[Theorem 1]{osuna} but is being used.}. 
	
	\begin{thm}\label{codimension.one.intro}
		Consider $g_0$ a unit-volume flat metric on $T^n$. There is $\mathcal U$ an  open neighborhood of $g_0$ in the smooth norm, so that if $g\in \mathcal U$ is such that $\vol_g(T^n)=1$ and, for every rational $(n-1)$-plane $P$,
		$$
		{{\bf M}(T_P)=||T_P||_{s,g}},
		$$
		then $g$ is isometric to $g_0$.
	\end{thm}
	{
		A direct corollary of Theorem \ref{higher.codimention.intro} and Theorem \ref{codimension.one.intro} is the local marked area spectrum rigidity of flat metrics on $T^n$. The \emph{marked $k$-area spectrum} for a metric $g$ is a function from $H_k(T^n,\Z)\to \R_{\geq0}$ sending an integer homology $[w]$ class to the $k$-area of the minimizer in $[w]$ under the metric $g$. Federer (Theorem \ref{thm.Federer}) observed that the marked $k$-area spectrum determines the stable norm on $H_k(T^n,\R)$. As a consequence, we have the following corollary.
		\begin{cor}
			Consider $g_0$ a unit-volume flat metric on $T^n$. There is $\mathcal U$ an open neighborhood of $g_0$ in the smooth norm, so that if $g\in \mathcal U$ and
			\begin{itemize}
				\item either $g$ has the same $k$-marked area spectrum as $g_0$ for some $1\leq k\leq n-2$;
				\item or $g$ has the same $n-1$-marked area spectrum as $g_0$ and $\vol_g(T^n)=\vol_{g_0}(T^n)$.
			\end{itemize}
			Then $g$ is isometric to $g_0$.    
		\end{cor}
	}
	
	The next theorem shows that the volume condition in Theorem \ref{codimension.one.intro} cannot be removed.
	{\begin{thm}\label{bangert.question} Consider a unit-volume flat metric $g_0$ on $T^3$. There is a smooth family of metrics $\{{g}(s)\}_{0\leq s<+\infty}$ so that
			\begin{itemize}
				\item ${g}(0)=g_0$ and  ${g}(s)$ is  not flat for $s>0$;
				\item the stable norm on $H_2(T^3,\R)$ remains fixed;
				\item  $\vol_{{g}(s)}(T^3)<1$ if $s\neq 0$ and $\lim_{s\to +\infty}\vol_{{g}(s)}(T^3)=0.$
			\end{itemize} 
	\end{thm}}
	The non-flat metrics $g(s)$ have the property that the volume can be very small but every primitive homology class admits a foliation by area-minimizing tori having the same area as in the Euclidean case.  Bangert in his ICM address \cite[p.~462]{bangert.icm.address} proposed the following problem.
	\medskip
	
	{\noindent{\bf Problem. }{\em Suppose $(T^n,g)$ is a Riemannian torus such that for every prime class $v\in H_{n-1}(T^n,\Z)$ there exists a foliation of $T^n$ by $g$-minimal $(n-1)$-tori in the class $v$. Is $g$ flat?}}
	\medskip
	
	The method (Theorem \ref{stable.warp}) used to show Theorem \ref{bangert.question} also answers the above problem negatively.
	{\begin{thm}\label{lots.deformations} There are non-flat metrics on $T^n$, $n\geq 3$, such that every primitive class in $H_{n-1}(T^n,\Z)$ admits a foliation by minimal tori and the stable unit sphere in $H_{n-1}(T^n,\R)$ is not $C^2$.
	\end{thm}}
	The theorem does not hold when $n=2$. Indeed, it follows from  \cite{bangert} that if the stable unit sphere in $H_{1}(T^2,\R)$  is $C^1$, then $g$ is flat and so the stable unit sphere is $C^{\infty}$.  To the best of our knowledge, Theorem \ref{lots.deformations} gives the first example of a smooth Riemannian metric on a closed manifold whose stable unit sphere is $C^1$ but not smooth.	
	
	Theorem \ref{lots.deformations} is reminiscent of Ambrozio--Marques--Neves \cite[Theorem C]{AMN25} in which the authors constructed non-round metrics on $S^n$ so that through every hyperplane there is a minimal hypersphere tangent to that hyperplane.
	
	The final result we prove concerns rigidity of the volume spectrum. To any closed Riemannian manifold $(M,g)$ we associate the volume spectrum $\{\omega_k(M,g)\}_{k\in\N}$ (see Section \ref{vs.defi} and \cite{lmn}) which can be thought of as a non-linear version of the Laplacian spectrum.
	
	We say that the volume spectra of two metrics $g$ and $g'$ on a closed manifold $M$ \emph{strongly coincide} if for every $k\in\mathbb N$ and every finite covering map $p:M'\to M$
	$$
	\omega_k(M',p^*g')=\omega_k(M',p^*g).
	$$
	\begin{thm}\label{spectrum.rigidity.intro} Consider  $g_0$ a unit-volume flat metric on $T^n$, $n\geq 2$. There is an open neighborhood $\mathcal U$ of $g_0$ in the smooth norm, so that if the volume spectrum of $g\in \mathcal U$  strongly coincides with the volume spectrum of $g_0$, then $g$ is isometric to $g_0$.
	\end{thm}
	The only other known case where rigidity of the volume spectrum is known is  $\RP^2$ with the round metric \cite[Theorem A]{AMN24}. 
	
	{The Laplacian spectrum of flat metrics on $T^n$ is rigid only if  $n\leq 3$ (from \cite[Theorem~A]{tanno-spectrum} and \cite{Schiemann90,Schiemann97}). Reasoning by analogy with this case, we would expect that, for low dimensions, the volume spectrum of flat metrics on $T^n$ is rigid but that rigidity could fail for high dimensions. We also expect the volume spectrum of flat metrics to be strongly rigid in every dimension.  Using Proposition \ref{long.cyclic.width}, this would follow if Theorem \ref{codimension.one.intro} held for any metric $g$.	
		
		\subsection{Remarks about the proofs}\label{remark.section}}
	
	An important ingredient in our proof is the \emph{normalized tensorial Radon transform} on $T^n$, which is the linearization of the area functional with respect to the metric. Fix a flat metric on $T^n$.  Given a symmetric $2$-tensor $b$ on $T^n$, for every  rational $d$-plane $P$ in $T^n$  we consider
	\begin{equation}\label{radon.definition}
		R_P(b):T^n\rightarrow \R, \quad R_P(b)(x)=\fint_{T_P}\tr _P b(x+y)\,dA(y),
	\end{equation}
	where $\tr_P b$ is the trace of the induced tensor $b|_P$. Related Radon transforms on tori were studied in \cite{Abouelaz11,AbouelazRouviere11,ilmavirta}, and the corresponding transform of symmetric $2$-tensors over Euclidean planes was studied in \cite{GrebnevStefanovUhlmannZhou26}.
	
	If $1\leq d\leq n-2$ and $b$ is a divergence-free symmetric $2$-tensor such that $R_P(b)=0$ for every rational $d$-plane $P$, then it is not difficult to see that $b=0$ (see Proposition \ref{injectivity.estimate}). Thus the linearization of Theorem \ref{higher.codimention.intro} has no kernel. The standard approach of arguing by contradiction and then using elliptic estimates to show that $|g-g_0|=O(|g-g_0|^2)$ does not work here. The reason is that the norm of the inverse of the Jacobi operator on $T_P$ can become arbitrarily large as $P$ varies, and so the implicit constant in $O(|g-g_0|^2)$ cannot be chosen uniformly. This is the main difficulty in the proof of Theorem \ref{higher.codimention.intro}. We outline our approach in Section \ref{basic.idea.higher.codimension}.
	
	The difficulty with Theorem \ref{codimension.one.intro} is different. There are divergence-free symmetric $2$-tensors $h\neq 0$ such that
	$$\int_{T^n} \tr h\,dV=0\quad\text{and}\quad R_P(h)=0\text{ for every rational hyperplane }P.$$
	Thus the linearized problem has a kernel. We use second-order integral identities in the Taylor expansions of volume and area to show that $g$ is isometric to $g_0$. We outline our approach in Section \ref{basic.idea.codimension.one}.
	
	Theorem \ref{bangert.question} and Theorem \ref{lots.deformations} are proven in Section \ref{non-rigid}. The idea behind the proof of Theorem \ref{lots.deformations} is simple. Given a metric $g$ on $T^n$, it is convenient to consider the symmetric $2$-tensor $G=\det(g)g^{-1}$. The reason is that if $P$ is a hyperplane orthogonal to a unit vector $v$, then $\det(g|_{P})=G(v,v)$, and so
	$$\text{area}_g(\Sigma)=\int_\Sigma\sqrt{G(\nu,\nu)}\,dA,$$
	where $\nu$ is the Euclidean unit normal to $\Sigma$.
	If $G=g(x_n)+dx_n^2$, where $g(x_n)$ is a flat metric on $T^{n-1}$, then a calibration argument shows that every primitive class in $H_{n-1}(T^n,\Z)$ is foliated by minimal tori (Theorem \ref{stable.warp}). If $n\geq 3$, we show that this ansatz produces non-flat metrics.
	
	The family of metrics in Theorem \ref{bangert.question} was found using one of the commercially available LLMs. It is an open problem whether $T^n$, for $n\geq 4$, admits a metric $g$ very close to a flat metric and having the same stable norm on $H_{n-1}(T^n,\R)$ as that flat metric. We expect such metrics to exist but the ansatz $G=g(x_n)+dx_n^2$ only works when $n=3$.
	
	Theorem \ref{spectrum.rigidity.intro} is proven in Section \ref{almgren-pitts-rigid}. The main computation has independent interest and we show that the volume spectrum of all finite covers of $T^n$ determines its stable norm on $H_{n-1}(T^n,\R)$. From the Weyl Law for the volume spectrum \cite{lmn} we know that the volume is also determined. Theorem \ref{spectrum.rigidity.intro}  is then a consequence of Theorem \ref{codimension.one.intro}.
	
	\subsection*{Acknowledgments} {A.S. thanks Jinxin Xue for helpful conversations on Hamiltonian dynamics and Ruixiang Zhang for helpful conversations on harmonic analysis.}
	
	\section{Notation}
	
	Recall that we say a $d$-plane $P\subset \R^n$ is rational if $P\cap \Z^n$ has rank $d$ {and that we consider the embedded subtorus of $T^n$  given by}
	$$
	T_P:=P/(P\cap \Z^n).
	$$
	
	{Fix a Euclidean metric $g_0$ on $T^n$. In Fourier expressions, $\xi=(\xi_1,\ldots,\xi_n)\in\Z^n$ denotes the integral covector $\sum_i\xi_i\,dx_i$, and we write $\xi(x)$ for its pairing with $x$. For ease of notation, whenever $\xi$ occurs in a metric expression, we identify it with the vector $\xi^{\sharp_{g_0}}$; it will be clear from the context whether $\xi$ denotes the covector or the corresponding vector; all geometric quantities are computed with respect to $g_0$, unless it is written otherwise.  For instance, $x\cdot y$ stands for $g_0(x,y)$ and $|v|$ for $|v|_{g_0}$. We use $O(q)$ to denote a quantity bounded by $C|q|$, where $C=C(n,d,g_0)$.}
	
	{Recall that the normalized trace Radon transform of a  symmetric $2$-tensor $b$ is 
		$$
		R_P(b)(x):=\fint_{T_P}\tr _P b(x+y)\,dA(y).
		$$
		where  $P$ is a rational $d$-plane and $\tr_P b$ is the trace of the induced tensor $b|_P$. The function $R_P(b)$ is constant in the $P$-direction.} 	
	Consider the Fourier decomposition
	$$
	b(x)=\sum_{\xi\in\Z^n}\widehat b(\xi)e^{2\pi i\xi(x)}.
	$$
	For every rational $d$-plane $P$, the following Fourier identity holds
	\begin{equation}\label{fourier.identity2}
		\widehat{R_P(b)}(\xi)
		=
		\begin{cases}
			\tr _P\widehat b(\xi), & \xi|_P=0,\\
			0, & \xi|_P\neq 0.
		\end{cases}
	\end{equation}
	Indeed,
	$$
	R_P(b)(x) = \fint_{T_P} \sum_{\xi\in\Z^n} \tr _P\widehat b(\xi)e^{2\pi i\xi(x+p)}\,dA(p),
	$$
	and
	$$
	\fint_{T_P}e^{2\pi i\xi(p)}\,dA(p)
	=
	\begin{cases}
		1, & \xi|_P=0,\\
		0, & \xi|_P\neq 0.
	\end{cases}
	$$
	We also use the Wiener norm
	$$
	||b||_{\mathcal W}:=\sum_{\xi\in\Z^n}|\widehat b(\xi)|,
	$$
	and the Wiener-Sobolev norm 
	$$||b||_{\mathcal W^s}=\sum_{\xi\in\Z^n}\langle\xi\rangle^s|\widehat b(\xi)|, \quad \langle\xi\rangle=(1+4\pi^2|\xi|^2)^{1/2}.$$
	We have $|b|_{L^\infty}\leq ||b||_{\mathcal W}$. If $m>s+n$ one has $||b||_{\mathcal W^s}\leq C|b|_{C^m}$ for some {$C=C(n,s,m,g_0)$ \cite[Corollary~3.3.10(a)]{grafakos-classical}}.\medskip

	Let $\text{Lip}_d(M)$ denote the set of all Lipschitz $d$-simplices on a closed manifold $M$. Given a metric $g$ on $M$ and $w\in H_d(M,\R)$ we set
	\begin{equation}\label{stable.norm.defi}
		||w||_{s,g}=\inf\left\{\sum_i |r_i|\text{vol}_g(\sigma_i):\sum_i r_i\sigma_i\in w, \sigma_i\in\text{Lip}_d(M), r_i\in\R \right\}.
	\end{equation}
	This defines the {\em stable norm} on $H_d(M,\R)$.    We use 
	$\mathcal{Z}_d(M,G)$ to denote the space of $d$-dimensional integral cycles with coefficients in $G=\Z$ or $\Z_2$ with the flat topology.  The mass of a current $\Sigma$ is denoted by ${\bf M}_g(\Sigma)$.
	
	Federer \cite[Section 5.8]{Federer74} discovered the following relationship between minimization over $\Z$-homology classes and over $\R$-homology classes. It also explains the term ``stable norm''.	
	\begin{thm}\label{thm.Federer}
		If $v\in H_d(M,\Z)$, then
		\begin{equation}\label{Federer.stable.norm}
			||v||_{s,g}=\inf\left\{\frac1k {\bf M}_g(S):S\in\mathcal Z_d(M,\Z),\, [S]=kv,\,{k\in\Z_+}\right\}.
		\end{equation}
	\end{thm}
	We abuse notation and use $||\Sigma||_{s,g}$ to denote $||[\Sigma]||_{s,g}$ when $\Sigma\in\mathcal Z_d(M,\Z)$.
	
	{For every rational $d$-plane $P$, $T_P\subset \T^n$ is calibrated and so ${\bf M}(T_P)=||T_P||_s$ (with respect to the fixed flat metric).}

	{\subsection{Volume spectrum}\label{vs.defi}
		
		The space $\mathcal{Z}_{n-1}(M,\mathbb{Z}_2)$ is weakly homotopically equivalent to $\mathbb{RP}^\infty$. Therefore the cohomology with $\mathbb{Z}_2$ coefficients is
		$$
		H^k(\mathcal{Z}_{n-1}(M,\mathbb{Z}_2), \mathbb{Z}_2)=\mathbb{Z}_2=\{0, \overline{\lambda}^k\},
		$$
		where $\overline{\lambda}$ is the generator of $H^1(\mathcal{Z}_{n-1}(M,\mathbb{Z}_2), \mathbb{Z}_2)=\mathbb{Z}_2$ and  $\overline{\lambda}^k$ is the $k$-th cup power of $\overline{\lambda}$.
		
		A $k$-sweepout is a continuous map $\Phi:X \rightarrow \mathcal{Z}_{n-1}(M,\mathbb{Z}_2)$ with no mass concentration (see \cite{mnindex} for the definition), where  $X$ is a finite-dimensional compact simplicial complex that depends on $\Phi$, such that
		$
		\Phi^*(\overline{\lambda}^k)\neq 0.
		$
		We denote $\Phi \in \mathcal{P}_k$. For any integer $k\geq 1$, the $k$-width of $(M,g)$ is the min-max number
		$$
		\omega_k(M,g)= \inf_{\Phi\in \mathcal{P}_k} \sup_{x\in {\rm dmn}(\Phi)} {\bf M}(\Phi(x)),
		$$
		where  ${\rm dmn}(\Phi)$ is the domain of $\Phi$.  
		
	}	
	\section{Proof of Theorem \ref{higher.codimention.intro}}
	
	The goal of this section is to prove the following result.
	\begin{thm}\label{higher.codimension.no.local.deformation2} 
		Consider $g_0$  a unit-volume flat metric on $T^n$, $1\leq d\leq n-2$, and  $r>7n$. There is $\delta>0$ such that the following holds:
		
		Let $g=g_0+h$, where $h$ is a divergence-free symmetric $2$-tensor for which, for every rational $d$-plane $P$,
		$$
		{{\bf M}(T_P)=||T_P||_{s,g}.}
		$$
		If $|h|_{C^r}<\delta$, then $h=0$.
	\end{thm}
	One immediate consequence of Theorem \ref {higher.codimension.no.local.deformation2} is {Theorem \ref{higher.codimention.intro}}\footnote{Using the Ebin-Slice Theorem in \cite[Proposition 11]{BM11} for instance.} 
	\begin{cor}\label{codimension.two.cor}
		Consider  $g_0$ a unit-volume flat metric on $T^n$ and $1\leq d\leq n-2$. There is $\mathcal U$ an  open neighborhood of $g_0$ in the smooth norm, so that if $g\in \mathcal U$ is such that for every rational $d$-plane $P$,
		$$
		{{\bf M}(T_P)=||T_P||_{s,g},}
		$$
		then $g$ is isometric to $g_0$.
	\end{cor}
	
	\subsection{Basic idea}\label{basic.idea.higher.codimension}  Assume that $g=g_0+h$ satisfies the hypothesis of Theorem \ref{higher.codimension.no.local.deformation2}.  If we show that $R_P(h)=0$ for all rational $P$ we see that $h=0$ (Proposition \ref{injectivity.estimate}). A Taylor expansion of the inequality ${\bf M}(T_P)\leq {\bf M}_{g}(T_P)$ shows that  (see \eqref{first.inequality.2})
	\begin{equation}\label{basic.idea.area}
		||R_P(h)||_{L^1}\lesssim \tr_P(K)+||h||_{\mathcal W}^2\quad\text{where }K=\int_{T^n}hdV.
	\end{equation}
	The main idea is to find a carefully chosen rational torus $T_P$ (Lemma \ref{choice.vector}) such that ${{\bf M}(T_P)=||T_P||_{s,g},}$ gives  $|K|\lesssim ||h||_{\mathcal W}^{3/2}$.  In that case we have
	$$||R_{Q}(h)||_{L^1}\lesssim \tr_Q(K)+||h||_{\mathcal W}^2\lesssim ||h||_{\mathcal W}^{3/2}$$
	for every rational plane $Q$. Using Lemma \ref{lemm.quadr.estimate} we deduce that $h=0$ if it is suitably  small.

	{We }explain how to prove the estimate for $|K|$ in the particular case where $$h=K+\mathcal{A} e^{2\pi i \xi_0(x)}+\bar{\mathcal{A}}  e^{-2\pi i \xi_0(x)}=K+\tilde h,\quad\xi_0\in\Z^n\setminus\{0\},$$  where $\mathcal{A}$ is a complex symmetric $2$-tensor. Choose a rational plane $P$  so that $|\xi_0^{\bot}|\lesssim|\xi_0^{\top}|$ and $|K|\lesssim|\tr_P K|$. The fact that this is possible is the content of   Lemma \ref{choice.vector}. 
	
	We already know that $\tr_P(K)\gtrsim -||h||^2_{\mathcal W}$ from {\eqref{basic.idea.area}}. Thus if we show that $\tr_P(K)\lesssim ||h||^2_{\mathcal W}$ we deduce that $|K|\lesssim||h||^2_{\mathcal W}$, which gives the desired estimate  for $|K|$.  
	
	{For the sake of the argument,  assume} there is an area-minimizing $\Sigma$ so that ${\bf M}_{g}(\Sigma)={\bf M}(T_P)$ and let $\Omega_P$ denote the calibration of $T_P$. A Taylor expansion of the area implies that 
	\begin{equation}\label{basic.idea.estimate.taylor}
		\int_\Sigma\tr _P h\,\Omega_P+\frac12\int_\Sigma |T_x\Sigma- P|^2dA\lesssim \int_\Sigma |h|^2\,dA.
	\end{equation}
	The condition that $|\xi_0^{\bot}|\lesssim|\xi_0^{\top}|$ implies (see \eqref{fourier.integration} and \eqref{fourier.integration2})
	$$\left|\int_\Sigma e^{2\pi i \xi_0(x)}\Omega_P\right| \lesssim \int_\Sigma |T_x\Sigma- P|dA$$ 
	and so
	$$
	\left| \int_\Sigma\tr _P \tilde h\,\Omega_P\right| \lesssim \int_\Sigma |\mathcal{A}||T_x\Sigma- P|dA.$$
	Therefore
	\begin{align*}
		\int_\Sigma\tr _P h\,\Omega_P&=\tr _P K {\bf M}(T_P)+\int_\Sigma\tr _P \tilde h\,\Omega_P\\
		&\gtrsim \tr _P K {\bf M}(T_P)- \int_\Sigma |\mathcal{A}||T_x\Sigma- P|dA.
	\end{align*}
	Inserting this estimate in \eqref{basic.idea.estimate.taylor} we obtain, after absorbing $|h||T_x\Sigma- P|$, that
	$$\tr_P(K)\lesssim \fint_\Sigma|h|^2dA+|\mathcal{A}|^2.$$
	This inequality implies $\tr_P(K)\lesssim ||h||_{\mathcal W}^2$.
	
	\subsection{Preliminaries}
	We start by deriving the necessary integral estimates. 
	
	\begin{prop}\label{injectivity.estimate}
		Let $1\leq d\leq n-2$. There is a constant $C=C(n,d,g_0)>0$ such that for every smooth divergence-free symmetric $2$-tensor $b$ we have
		$$
		||b||_{L^2}^2\leq C\left(\sup_{\text{rational }d\text{-plane }P}||R_P(b)||_{L^1(T^n)}\right)||b||_{\mathcal W}.
		$$
		In particular, if $R_P(b)=0$ for every rational $d$-plane $P$, then $b=0$.
	\end{prop}
	\begin{proof}
		{Set}
		$$
		M:=\sup_{\text{rational }d\text{-plane }P}||R_P(b)||_{L^1(T^n)}.
		$$
		Whenever $\xi|_P=0$, the Fourier identity gives
		\begin{equation}\label{fourier.identity}
			|\tr _P\widehat b(\xi)|
			=
			|\widehat{R_P(b)}(\xi)|
			\leq
			||R_P(b)||_{L^1(T^n)}
			\leq M.
		\end{equation}
		
		Let $E$ be an $N$-dimensional Euclidean vector space,  $1\leq d\leq N-1$, and $G_d(E)$ the Grassmannian of $d$-planes. If $S$ is a  symmetric $2$-tensor on $E$ (possibly complex-valued), write
		$$
		S=\frac{\tr _E S}{N}\text{Id}+S_0,  \quad\text{where } \tr _E S_0=0.
		$$
		Then, with $\mu$ being the normalized $O(E)$-invariant probability measure on $G_d(E)$, we have \cite[Theorem~2.4]{bodmann-ehler-graf}
		\begin{equation}\label{integration.formula}
			\int_{G_d(E)}
			|\tr _P S|^2\,d\mu(P)
			=
			\frac{d^2}{N^2}|\tr _E S|^2
			+
			\frac{2d(N-d)}{N(N-1)(N+2)}|S_0|^2.
		\end{equation}
		In particular, since $1\leq d\leq N-1$, both coefficients are positive. Hence there is $C=C(N,d)>0$ such that
		\begin{equation}\label{integration.grass}
			\int_{G_d(E)}
			|\tr _P S|^2\,d\mu(P)
			\geq
			C|S|^2.
		\end{equation}
		Consequently,
		\begin{equation}\label{grass.identity}
			|S|
			\leq
			C^{-1/2}
			\sup_{P\in G_d(E)}|\tr _P S|.
		\end{equation}
		
		We apply this estimate to the Fourier coefficients of $b$. First consider the zero mode. Since $0\perp P$ for every $P$, \eqref{fourier.identity} gives
		$$
		|\tr _P\widehat b(0)|\leq M
		$$
		for every rational $d$-plane $P\subset\R^n$. Rational $d$-planes are dense in $G_d(\R^n)$ and  so, by continuity,
		$$
		|\tr _P\widehat b(0)|\leq M
		$$
		for every real $d$-plane $P\subset\R^n$. Applying \eqref{grass.identity} with $E=\R^n$ gives, for some other $C=C(n,d)$,
		$$
		|\widehat b(0)|\leq C M.
		$$
		
		Now fix $\xi\in\Z^n\setminus\{0\}$, and set $E=\ker\xi$. Then $\dim E=n-1$. Since $1\leq d\leq n-2$, we have $1\leq d\leq \dim E-1$. For every rational $d$-plane $P\subset E$, we have $\xi|_P=0$, and hence
		$$
		|\tr _P\widehat b(\xi)|\leq M.
		$$
		Rational $d$-planes are dense in $G_d(E)$ {since $E$ is rational} and  so
		$$
		|\tr _P\widehat b(\xi)|\leq M
		$$
		for every real $d$-plane $P\subset E$.
		
		Since $b$ is divergence-free, with the convention $(\delta b)_j=-\partial_i b_{ij}$, we have
		$$
		\xi_i\widehat b_{ij}(\xi)=0.
		$$
		Equivalently, $\widehat b(\xi)(\xi,\cdot)=0$. Thus $\widehat b(\xi)$ is determined by its restriction to $E=\ker\xi$. Applying \eqref{grass.identity} on $E$ to $S=\widehat b(\xi)|_E$, we get
		$$
		|\widehat b(\xi)|\leq C M.
		$$
		Therefore $|\widehat b(\xi)|\leq C M$ for every $\xi\in\Z^n$.
		
		Finally, by Plancherel and the definition of $\mathcal W$,
		$$
		||b||_{L^2}^2 = \sum_{\xi\in\Z^n}|\widehat b(\xi)|^2 \leq \left(\sup_{\xi\in\Z^n}|\widehat b(\xi)|\right) \sum_{\xi\in\Z^n}|\widehat b(\xi)| \leq C M||b||_{\mathcal W}.
		$$
	\end{proof}

	We use the following standard interpolation estimate. If $s>0$, then
	\begin{equation}\label{interpolation}
		||b||_{\mathcal W}\leq C||b||_{L^2}^{1-\theta}||b||_{\mathcal W^s}^{\theta},\qquad
		\theta=\frac{n/2}{s+n/2},
	\end{equation}
	{for some $C=C(n,s,g_0)$.} Indeed, for any $L\geq1$,
	$$
	\sum_{{\xi\in\Z^n}}|\widehat b(\xi)|
	\leq
	\sum_{|\xi|\leq L}|\widehat b(\xi)|+\sum_{|\xi|>L}|\widehat b(\xi)|.
	$$
	By Cauchy-Schwarz and Plancherel,
	$$
	\sum_{|\xi|\leq L}|\widehat b(\xi)|\leq CL^{n/2}||b||_{L^2}.
	$$
	Also,
	$$
	\sum_{|\xi|>L}|\widehat b(\xi)|\leq L^{-s}||b||_{\mathcal W^s}.
	$$
	Choosing
	$$
	L=\left(\frac{||b||_{\mathcal W^s}}{||b||_{L^2}}\right)^{1/(s+n/2)}
	$$
	gives
	$$
	||b||_{\mathcal W}\leq C||b||_{L^2}^{1-\theta}||b||_{\mathcal W^s}^{\theta},\qquad
	\theta=\frac{n/2}{s+n/2}.
	$$
	
	The interpolation estimate we just discussed implies the following result.
	
	\begin{lemm}\label{lemm.quadr.estimate}
		Given $C_0$ and $r>3n$, there is {$\delta=\delta(C_0,n,d,r,g_0)>0$} with the following property. If the divergence-free symmetric $2$-tensor $b$ is such that
		$$
		\sup_{\text{rational }d\text{-plane }P}||R_P(b)||_{L^1(T^n)}\leq C_0||b||_{\mathcal W}^{3/2}
		$$
		and $|b|_{C^r}<\delta$, then $b=0$.
	\end{lemm}
	\begin{proof}
		In what follows, $C$ denotes a generic constant depending only on $C_0,d,n,r$ and $g_0$.
		From Proposition \ref{injectivity.estimate} and the hypothesis,
		$$
		||b||_{L^2}^2\leq C||b||_{\mathcal W}^{5/2}.
		$$
		Set $s=2n$. The interpolation inequality \eqref{interpolation} gives
		$$
		||b||_{\mathcal W}\leq C||b||_{L^2}^{1-\theta}||b||_{\mathcal W^s}^\theta,\qquad \theta=\frac{1}{5}.
		$$
		Therefore
		$$
		||b||_{L^2}^2\leq C||b||_{L^2}^{\frac52(1-\theta)}||b||_{\mathcal W^s}^{\frac52\theta}=C||b||_{L^2}^2||b||_{\mathcal W^s}^{1/2}.
		$$
		If $b\neq0$, dividing by $||b||_{L^2}^2$ gives
		$
		||b||_{\mathcal W^s}\geq C.
		$ We have $||b||_{\mathcal W^s}$ bounded by $|b|_{C^r}$ and so $b=0$ if $|b|_{C^r}$  is small.
	\end{proof}

	\subsection{Proof of Theorem \ref{higher.codimension.no.local.deformation2}}
	\begin{proof}
		For every $L\in \N$ we set
		$$ 
		K=\hat h(0),\quad h_L=\sum_{{0\neq\xi}\in\Z^n,\, |\xi|\leq L}\widehat h(\xi)e^{2\pi i \xi(x)},\quad E_L=\sum_{{|\xi|>L}}\widehat h(\xi)e^{2\pi i \xi(x)}
		$$
		so that
		$$
		h=K+h_L+E_L.
		$$
		We note that $K$ is a constant real symmetric $2$-tensor. We will repeatedly use the following upper bound
		$$|h|(x)\leq ||h||_{\mathcal W} \quad\text{and}\quad |E_L|(x)\leq ||E_L||_{\mathcal W}\text{ for all }x\in T^n.$$
		If $||h||_{\mathcal W}=0$, then $h=0$. Thus we assume $||h||_{\mathcal W}>0$.
		
		Since every translate $x+T_P$ is an admissible competitor in the class $[T_P]$, the hypothesis {${\bf M}(T_P)=||T_P||_{s,g}$
			implies (see \eqref{Federer.stable.norm}) that}
		$$
		{\bf M}_g(x+T_P)\geq {\bf M}(T_P).
		$$
		We assume $|h|\leq 1$. {With respect to $g_0$-orthonormal coordinates we have
			$
			dA_g=\sqrt{\det(I+h|_P)}\,dA
			$
			and
			$$
			\sqrt{\det(I+h|_P)} =1+\frac12\tr_P h+O(|h|^2).
			$$}Thus 
		$$
		{\bf M}_g(x+T_P)\leq {\bf M}(T_P)\left(1+\frac12R_P(h)(x)+ \frac{C_0}{2}R_P(|h|^2)(x)\right),
		$$
		{for some $C_0=C_0(n,d,g_0)$.}
		Therefore we have for  {every} rational $d$-plane $P$ and $x\in T^n$,
		\begin{equation}\label{first.inequality}
			R_P(h)(x)\geq -C_0 R_P(|h|^2)(x)\geq -C_0||h||^2_{\mathcal W}.
		\end{equation}

		The negative part of $R_P(h)$ satisfies $(R_P(h))_-\leq C_0||h||^2_{\mathcal W}$ and so
		\begin{multline}\label{first.inequality.2}
			\int_{T^n}|R_P(h)|dV=\int_{T^n}R_P(h)dV+2\int_{T^n}(R_P(h))_-dV\\ 
			\leq\int_{T^n}R_P(h)dV +2C_0||h||^2_{\mathcal W}=\tr _P(K) +2C_0||h||^2_{\mathcal W}
		\end{multline}
		for every rational $d$-plane $P$. Next, we show that $|K|$ can be bounded by $||h||^{3/2}_{\mathcal W}$ so that Lemma \ref{lemm.quadr.estimate} can be applied {to conclude that $h=0$.}

		In order to estimate $|K|$  we start by choosing a rational $d$-plane depending on $h$. Given a $d$-plane $P$ and a vector $v\in \R^n$, we denote by $v^{\bot}$ and $v^{\top}$, respectively, the orthogonal and tangential projections on $P$. We also consider the following set
		$$
		\Xi_L=\{\xi\in\Z^n\setminus\{0\}:|\xi|\leq L\}.
		$$

		\begin{lemm}\label{choice.vector} 
			There is {$c_0=c_0(d,n,g_0)>0$} so that for every $L\in \N$ and every symmetric $2$-tensor $S$, there is a rational $d$-plane $P$ such that
			$$
			|\xi^{\top}|\geq c_0L^{-n}|\xi^{\bot}|\quad\text{for every } \xi\in \Xi_L,
			$$
			and
			$$
			|\tr _P S|\geq c_0|S|.
			$$
		\end{lemm}
		
		\begin{proof}
			It suffices to consider the case where  $S\neq0$. Without loss of generality, we assume that $|S|=1$.
			
			From \eqref{integration.formula} we find $a=a(d,n)>0$ such that 
			$$
			\int_{G_d(\R^n)}|\tr _P S|^2d\mu(P)\geq a.
			$$
			Set $\alpha=\sqrt{a/2}$ and consider $
			\mathcal G_S=\{P\in G_d(\R^n):|\tr _P S|>\alpha\}
			$. We have $|\tr _P S|\leq d$ and so 
			$$
			a\leq \int_{G_d(\R^n)}|\tr _P S|^2d\mu\leq \alpha^2+d^2\mu(\mathcal G_S).
			$$
			Thus 
			$
			m_0=\frac{a}{2d^2}\leq \mu(\mathcal G_S).
			$
			
			There is a constant $C_0=C_0(d,n)$ such that for every unit vector $e$ and every $0<\rho<1$,
			$$
			\mu\{P\in G_d(\R^n):|e^{\top}|\leq \rho\}\leq C_0\rho.
			$$
			This follows by considering the function $P\in G_d(\R^n)\mapsto f(P)=|e^{\top}|$ and noticing that, with respect to the homogeneous metric on $G_d(\R^n)$,  the sub-Grassmannian $f^{-1}(0)=G_d(e^\perp)$ is totally geodesic and $f(P)=\sin\text{dist}(P,f^{-1}(0))$ \cite[Theorem~6]{Ye16}. Thus the volume of the set is controlled by the volume of a tubular neighborhood of a totally geodesic submanifold.
			
			We deduce from the estimate above that{, for some $C_1=C_1(n,g_0)$} 
			$$
			\sum_{\xi\in\Xi_L}\mu\{P\in G_d(\R^n):|\xi^\top|\leq \tau L^{-n}|\xi^\bot|\}\leq C_0\tau L^{-n}|\Xi_L|\leq C_0C_1\tau.
			$$
			Choosing $\tau=\tau(n,d,g_0)$ so that {$C_0C_1\tau<m_0/2$,}  we can find $P\in \mathcal G_S$ so that
			$$
			|\xi^\top|>\tau L^{-n}|\xi^\bot|\quad\text{for every }\xi\in\Xi_L.
			$$
			Since all the inequalities are strict and rational $d$-planes are dense in $G_d(\R^n)$, we may choose $P$ rational. Setting $c_0=\min\{\alpha,\tau\}$ gives the result.
		\end{proof}

		In what follows, $C$ denotes a generic constant that depends  on $n$, {$d$, and $g_0$.}
		Consider the rational plane $P$ given by Lemma \ref{choice.vector} with  $S=K$. {From \eqref{first.inequality.2}  we have}
		$$
		\tr _P K\geq -C||h||^2_{\mathcal W}.
		$$
		{From \eqref{Federer.stable.norm} we find $k\in\N$ and an integral $d$-cycle $\Sigma$ homologous to $kT_P$ and such that} 
		\begin{equation}\label{area.comparison}
			{\bf M}(kT_P)\leq {\bf M}_g(\Sigma)\leq (1+||h||_{\mathcal W}^2){\bf M}(kT_P).
		\end{equation}
		We regard $\Sigma$ as an integer rectifiable cycle and denote by $\tau(x)$ the oriented tangent plane at $x$ which is defined almost everywhere. We denote by $\vec{\tau}(x)$ the associated oriented $g_0$-unit simple $d$-vector and by $\vec{P}$ the corresponding oriented $d$-vector for $T_P+x$. We have
		$$
		{\bf M}_g(\Sigma)=\int_\Sigma\left(1+\frac12\tr _\tau h+O(|h|^2)\right)\,dA.
		$$
		There is a constant $C$ so that for every symmetric $2$-tensor $S$ and every oriented $d$-plane $Q$ we have
		$$
		|\tr _Q S-\tr _P S|\leq C|S||Q-P|\leq C|S||\vec Q-\vec P|
		$$
		and so
		$$
		{\bf M}_g(\Sigma)={\bf M}(\Sigma)+\int_\Sigma \frac12\tr _P h\,dA+\int_\Sigma O(|h|^2+|h||\vec \tau-\vec P|)\,dA.
		$$
		Consider $\Omega_P$, the closed $d$-form that calibrates $T_P$ and is described in \cite[Example~II, p.~58]{HarveyLawson82}.  It has the property that $\Omega_P(\vec\tau)=\langle\vec\tau,\vec P\rangle$ (with respect to the natural inner product induced on $\Lambda^d(\R^n)$).  Since $|\vec\tau|=|\vec P|=1$, we have
		$$
		1-\Omega_P(\vec\tau)=\frac{1}{2}|\vec \tau-\vec P|^2.
		$$
		Hence
		$$
		{\bf M}_g(\Sigma)={\bf M}(\Sigma)+\int_\Sigma \frac12\tr _P h\,\Omega_P+\int_\Sigma O(|h|^2+|h||\vec \tau-\vec P|)\,dA
		$$
		and (using $[\Sigma]={[kT_P]}$)
		$${\bf M}(\Sigma)=\int_\Sigma\Omega_P+\frac12\int_\Sigma |\vec \tau-\vec P|^2dA={{\bf M}(kT_P)}+\frac12\int_\Sigma |\vec \tau-\vec P|^2dA.$$
		Therefore, absorbing $|h||\vec \tau-\vec P|$ into $|h|^2$ and  $|\vec \tau-\vec P|^2$, we obtain from the two identities above
		$$
		\frac12\int_\Sigma\tr _P h\,\Omega_P+\frac14\int_\Sigma |\vec \tau-\vec P|^2dA\leq {\bf M}_g(\Sigma)-{\bf M}(kT_P)+C\int_\Sigma |h|^2\,dA.
		$$
		
		Furthermore
		$\tr _P h=\tr _P K+\tr _P h_L+\tr _P E_L$ and so we obtain from \eqref{area.comparison}
		
		\begin{equation}\label{second.inequality}
			\begin{aligned}
				&\frac12\int_{\Sigma}\tr _P K\,\Omega_P+\frac12\int_{\Sigma}\tr _P h_L\,\Omega_P\\
				&\quad+\frac14\int_\Sigma |\vec \tau-\vec P|^2dA\leq C\int_{\Sigma}\left(||h||_{\mathcal W}^2+||E_L||_{\mathcal W}\right)\,dA.
			\end{aligned}
		\end{equation}
		
		We now estimate the oscillatory part  $\tr _P h_L\,\Omega_P$. We start with the following identity. Consider $\xi\in \Z^n$ with $\xi^{\top}\neq 0$ and set $e=\xi^{\top}/|\xi^{\top}|$. Then, {with $\alpha:=\iota_e\Omega_P$,} we have 
		\begin{equation}\label{calibration.fourier} 
			e^{2\pi i \xi(x)}\Omega_P=\frac{1}{2\pi i|\xi^{\top}|}d(e^{2\pi i \xi(x)}\alpha)-\frac{1}{|\xi^{\top}|}e^{2\pi i \xi(x)}d(x\cdot\xi^{\bot})\wedge \alpha.
		\end{equation}
		Indeed, if $x_e=x\cdot e$, then $\Omega_P=dx_e\wedge \alpha$ and so
		\begin{multline*}
			d(e^{2\pi i \xi(x)}\alpha)=2\pi i e^{2\pi i \xi(x)}d(\xi(x))\wedge \alpha=2\pi i e^{2\pi i \xi(x)}\left(d(x\cdot\xi^{\top})+d(x\cdot\xi^{\bot})\right)\wedge \alpha\\
			=2\pi i e^{2\pi i \xi(x)}|\xi^{\top}|dx_e\wedge \alpha+2\pi i e^{2\pi i \xi(x)}d(x\cdot\xi^{\bot})\wedge \alpha\\
			=2\pi i e^{2\pi i \xi(x)}|\xi^{\top}|\Omega_P+2\pi i e^{2\pi i \xi(x)}d(x\cdot\xi^{\bot})\wedge \alpha.
		\end{multline*}
		Since $\partial\Sigma=0$, the exact term in \eqref{calibration.fourier} integrates to zero over $\Sigma$. Assuming $\xi^{\bot}\neq 0$ and setting $v=\xi^{\bot}/|\xi^{\bot}|$, $x_v=x\cdot v$, we obtain
		\begin{equation}\label{fourier.integration}
			\int_\Sigma e^{2\pi i \xi(x)}\Omega_P=-\frac{|\xi^{\bot}|}{|\xi^{\top}|}\int_{\Sigma}e^{2\pi i \xi(x)}dx_v\wedge \alpha.
		\end{equation}
		If $\xi^{\bot}=0$, the integral above is zero. 
		Otherwise, recalling Lemma \ref{choice.vector}, we deduce
		$$
		\left|\int_\Sigma e^{2\pi i \xi(x)}\Omega_P\right|\leq \frac{|\xi^{\bot}|}{|\xi^{\top}|}\int_{\Sigma}|dx_v\wedge \alpha(\vec\tau)|dA\leq \frac{L^n}{c_0}\int_{\Sigma}|dx_v\wedge \alpha(\vec\tau)|dA.
		$$
		We have $dx_v\wedge \alpha(\vec P)=0$ and so we find $C$ so that  for every oriented $d$-plane $Q$
		\begin{equation}\label{fourier.integration2}
			|dx_v\wedge \alpha(\vec Q)|\leq C|\vec Q-\vec P|.
		\end{equation}
		Hence
		\begin{align*}
			\left|\int_{\Sigma}\tr _P h_L\,\Omega_P\right|&\leq\sum_{\xi\in\Xi_L}\left|\tr _P\hat h(\xi)\int_\Sigma e^{2\pi i \xi(x)}\Omega_P \right|\\
			&\leq CL^n\!\sum_{\xi\in\Xi_L}\!|\hat h(\xi)|\!\int_{\Sigma}|\vec \tau-\vec P|dA\!=\!CL^n\!||h_L||_{\mathcal W}\!\int_{\Sigma}|\vec \tau-\vec P|dA\\
			&\leq \frac14\int_\Sigma |\vec \tau-\vec P|^2dA+CL^{2n}\int_{\Sigma}||h||_{\mathcal W}^2dA.
		\end{align*}
		Inserting this estimate into \eqref{second.inequality}  (and using $\int_\Sigma\Omega_P={\bf M}(kT_P))$ we obtain
		$$
		\tr _PK\leq C\left(L^{2n}||h||_{\mathcal W}^2+||E_L||_{\mathcal W}\right) .
		$$
		We had already shown that $\tr _PK\geq -C||h||^2_{\mathcal W}$ and thus
		$$|\tr _PK|\leq C\left(L^{2n}||h||_{\mathcal W}^2+||E_L||_{\mathcal W}\right).$$
		{Set  $s=6n$. From the fact that $|h|$ is small in $C^r$ with $r>7n$ we assume that $||h||_{\mathcal W^s}<1.$} We choose  $L=\lfloor ||h||_{\mathcal W}^{-1/(4n)}\rfloor$. In this case
		$$L^{2n}||h||_{\mathcal W}^2\leq ||h||_{\mathcal W}^{3/2}$$
		and
		$$
		\begin{aligned}
			||E_L||_{\mathcal W}&=\sum_{{|\xi|>L}}|\hat h(\xi)|\leq C L^{-s} ||E_L||_{\mathcal W^s}\\
			&\leq C ||h||_{\mathcal W}^{s/(4n)}||h||_{\mathcal W^s}\leq C||h||_{\mathcal W}^{3/2}||h||_{\mathcal W^s}\leq C||h||_{\mathcal W}^{3/2}.
		\end{aligned}
		$$
		Hence, recalling that $|\tr _PK|\geq c_0|K|$, we obtain
		$$
		|K|\leq c_0^{-1}|\tr _PK|\leq C||h||_{\mathcal W}^{3/2}.
		$$
		Inserting the estimate $|K|\leq C||h||_{\mathcal W}^{3/2}$ in \eqref{first.inequality.2}, we find some constant $C=C(d,n,g_0)$ so that
		$$\int_{T^n}|R_Q(h)|dV\leq  C||h||_{\mathcal W}^{3/2}$$
		for every rational $d$-plane $Q$.
		Lemma \ref{lemm.quadr.estimate} now applies and gives $h=0$ if $|h|_{C^r}$ is sufficiently small.
	\end{proof}

	\section{No local deformation in codimension one}\label{local.deformation.section}
	
	In this section we prove the following theorem.
	\begin{thm}\label{codimension.one.thm}
		Consider  $g_0$ a unit-volume flat metric on $T^n$. There is $\delta>0$ such that the following holds: 
		Let $g=g_0+h$, where $h$ is a divergence-free symmetric $2$-tensor with $\vol_g(T^n)=1$ and such that, for every rational $(n-1)$-plane $P$,
		$$
		{{\bf M}(T_P)=||T_P||_{s,g}.}
		$$
		If $|h|_{C^{n+10}}<\delta$, then $h=0$.
	\end{thm}
	One immediate consequence of  Theorem \ref{codimension.one.thm} is Theorem \ref{codimension.one.intro}\footnote{Using the Ebin-Slice Theorem in \cite[Proposition 11]{BM11} for instance.} 
	\begin{cor}\label{codimension.one.cor}
		Consider  $g_0$ a unit-volume flat metric on $T^n$. There is $\mathcal U$ an  open neighborhood of $g_0$ in the smooth norm, so that if $g\in \mathcal U$ is such that $\vol_g(T^n)=1$ and, for every rational $(n-1)$-plane $P$,
		$$
		{{\bf M}(T_P)=||T_P||_{s,g}},
		$$
		then $g$ is isometric to $g_0$.
	\end{cor}
	\subsection{Basic idea}\label{basic.idea.codimension.one} Consider $g$ satisfying the hypothesis of Theorem \ref{codimension.one.intro} and suppose $g=h+g_0$ where $h$ is divergence-free. It is part of the theory \cite[Theorem 5.1]{cgn} that  every  primitive homology class $[T_P]\in H_{n-1}(T^n,\Z)$ is foliated by minimal tori $\{\Sigma_x\}_{x\in T^n}$, where $\Sigma_x=\Sigma_{x'}$ if $x'\in T_P+x$. Say that $\nu$ is a unit normal to $P$ and $f_P$ is a function on $T^n$ so that 
	$$\Sigma_x=\{z+f_P(z)\nu: z \in T_P+x\}.$$
	A Taylor expansion of ${\bf M}(T_P)=\int_{T^n}{\bf M}_{g}(\Sigma_x)dV(x)$ gives (see \eqref{U.identity})
	\begin{equation}\label{key.identity}
		\frac12\int_{T^n}|\nabla_P f_P|^2dV\simeq \frac12\int_{T^n}\tr _{P} h\,dV+\int_{T^n}\left[\frac18(\tr _{P}h)^2-\frac14|h^\top|^2\right]dV.
	\end{equation}
	{Denote the  right-hand side of \eqref{key.identity} by $V(P)$.} It  can easily be extended to all  $P\in G_{n-1}(\R^n)$ and we have
	$$\int_{G_{n-1}(\R^n)}V(P)d\mu(P)=\int_{T^n}\frac{n-1}{2n}\tr h+\frac{n^2-5}{8n(n+2)}|\tr h|^2-\frac{n^2-3}{4n(n+2)}|h|^2dV.$$
	From $\vol_g(T^n)=1$ one deduces
	$$\frac12\int_{T^n}\tr h\,dV\simeq\frac14\int_{T^n}|h|^2dV-\frac18\int_{T^n}|\tr h|^2dV$$
	and so
	$$\int_{G_{n-1}(\R^n)}V(P)d\mu(P)\simeq\int_{T^n}\frac{n+1}{4n(n+2)}|h|^2-\frac{n+3}{8n(n+2)}|\tr h|^2dV.$$
	Denote the left-hand side by $W(P)$. {Using  the fact that $f_P$ is an approximate solution to a degenerate elliptic equation} we show that  (up to second order) $P\mapsto W(P)$ admits an extension to all hyperplanes $P$ in the Grassmannian $G_{n-1}(\R^n)$ and
	$$\int_{G_{n-1}(\R^n)}W(P)d\mu(P)\simeq \int_{T^n}\frac{|\tr h|^2+2|h|^2}{8n(n+1)(n+2)}dV.$$
	This identity  follows from Lemma \ref{fourier.lemma} and \eqref{integration.grassmanian}. We have from \eqref{key.identity} that
	$$\int_{G_{n-1}(\R^n)}W(P)d\mu(P)\simeq\int_{G_{n-1}(\R^n)}V(P)d\mu(P)$$
	and thus
	$$\int_{T^n}\frac{|\tr h|^2+2|h|^2}{8n(n+1)(n+2)}dV\simeq \int_{T^n}\frac{n+1}{4n(n+2)}|h|^2-\frac{n+3}{8n(n+2)}|\tr h|^2dV.$$
	This estimate implies that $h=0$.
	
	\subsection{Preliminaries}
	We denote by $\Delta_P$ and $\nabla_P$, respectively, the leafwise Laplacian restricted to the hyperplane $P$ and the tangential projection of the gradient on $P$. We denote by $v^{\bot}$ and $v^{\top}$, respectively, the orthogonal and tangential projections on $P$. When $P$ is rational, we also consider
	$$\Pi^0_P(\alpha)(x)=\alpha(x)-\fint_{T_P+x}\alpha dA, \quad\alpha\in C^{\infty}(T^n).$$
	The next lemma will be used to estimate one of the main terms. The elliptic equation in the lemma appears naturally when linearizing the graphical minimal surface equation with respect to the ambient metric.
	\begin{lemm}\label{fourier.lemma}
		Let $h$ denote a smooth { symmetric $2$-tensor}. Given a rational hyperplane $P$ with normal $\nu$, set $a=h(\nu,\nu)\in C^{\infty}(T^n)$. There is $f\in C^\infty(T^n)$ so that
		$$\Delta_P f=\frac12\Pi_P^0\partial_\nu a\quad\text{and}\quad \int_{T_P+x}fdA=0\text{ for all }x\in T^n.$$
		Moreover
		$$\int_{T^n}|\nabla_P f|^2dV=\frac14\sum_{\xi \in \Z^n,\xi^{\top}\neq0} \frac{|\xi^{\bot}|^2}{|\xi^{\top}|^2}|\widehat h(\xi)(\nu,\nu)|^2.
		$$
	\end{lemm}
	\begin{proof} 
		Consider the Fourier decomposition
		$$
		h(x)=\sum_{\xi\in \Z^n}\widehat h(\xi)e^{2\pi i\xi(x)}.
		$$
		We then have
		\begin{align*}\partial_\nu a(x)&=\sum_{0\neq \xi\in \Z^n}2\pi i \xi(\nu) \widehat h(\xi)(\nu,\nu)e^{2\pi i\xi(x)}\\
			&{=\sum_{ \xi\in \Z^n,\xi^{\top}\neq 0}2\pi i \xi(\nu) \widehat h(\xi)(\nu,\nu)e^{2\pi i\xi(x)}}.
		\end{align*}
		Consider 
		$$f(x)=-\sum_{ \xi\in \Z^n,\xi^{\top}\neq 0}\frac{ i\xi(\nu)}{4\pi|\xi^{\top}|^2}\widehat h(\xi)(\nu,\nu)e^{2\pi i\xi(x)}.$$
		We have 
		$$\Delta_Pe^{2\pi i\xi(x)}=-4\pi^2|\xi^{\top}|^2e^{2\pi i\xi(x)}$$
		and so $f$ satisfies the desired equation.  We have from integration by parts and  {$|\xi^{\bot}|^2=\xi(\nu)^2$}  that
		$$\int_{T^n}|\nabla_P f|^2dV=-\frac12\int_{T^n}f\partial_\nu adV=\frac14\sum_{\xi \in \Z^n,\xi^{\top}\neq0} \frac{|\xi^{\bot}|^2}{|\xi^{\top}|^2}|\widehat h(\xi)(\nu,\nu)|^2.
		$$
		We address the regularity of $f$. {Consider $p\in\Z^n$ so that $\nu=p/|p|$. For all $\xi\in \Z^n$ so that $\xi^{\top}\neq0$ we have
			$$|p|^2|\xi^{\top}|^2=|p|^2|\xi|^2-|\xi\cdot p|^2\geq 1$$
			and so $|\xi^{\top}|^2	\geq |p|^{-2}$.}
		Thus we obtain for all $m\in\N$ that
		$$
		\sum_{ \xi\in \Z^n,\xi^{\top}\neq 0}\frac{ |\xi|^m{\xi(\nu)^2}}{|\xi^{\top}|^4}|\widehat h(\xi)(\nu,\nu)|^2{\leq|p|^4{\sum_{ \xi\in \Z^n-\{0\}} |\xi|^{m+2}|\widehat h(\xi)|^2}} <+\infty
		$$
		and so $f$ is smooth. It is clear that $f$ has vanishing mean on any $T_P+x$. 
	\end{proof}

	We say a unit vector $\nu$ {\em rational} when its orthogonal hyperplane $\nu^{\bot}$ is rational. The next lemma is relevant to substitute integration over all directions in $H_{n-1}(T^n,\R)$ by a finite weighted sum of rational vectors in the unit sphere.
	
	\begin{lemm}\label{finite.rational.average}
		For every $\varepsilon>0$ there are rational unit vectors $\nu_1,\ldots,\nu_N\in S^{n-1}\subset \R^n$ and positive numbers $\lambda_1,\ldots,\lambda_N$ with $\sum_j\lambda_j=1$ such that the following statements hold:
		For every real symmetric $2$-tensor $S$ on $\mathbb R^n$,
		$$
		\left|\sum_j\lambda_j(\tr_{\nu_j^{\bot}}S)-\frac{n-1}{n}\tr S\right|\leq\varepsilon|S|,
		$$
		and
		$$
		\left|\!\sum_j\lambda_j\left[\frac18\bigl(\tr_{\nu_j^{\bot}}S\bigr)^2-\!\frac14{\left|S\right|_{\nu_j^{\bot}}}^2\right]-\!\frac{n^2-5}{8n(n+2)}(\tr S)^2+\!\frac{n^2-3}{4n(n+2)}|S|^2\!\right|\leq\varepsilon|S|^2.
		$$
		Moreover, if $e\in S^{n-1}$ and $S$ is a complex symmetric $2$-tensor satisfying $S(e,\cdot)=0$, then
		$$
		\sum_j\lambda_j\Psi(e,S,\nu_j)\leq\frac{|\tr S|^2+2|S|^2}{4n(n+1)(n+2)}+\varepsilon|S|^2,
		$$
		where
		\begin{equation}\label{psi.definition}
			\Psi(e,S,\nu)=
			\begin{cases}
				\displaystyle\frac14\frac{(e\cdot\nu)^2}{1-(e\cdot\nu)^2}|S(\nu,\nu)|^2, & \nu\neq\pm e,\\
				0, & \nu=\pm e.
			\end{cases}
		\end{equation}
	\end{lemm}
	\begin{proof}
		Denote by $d\nu$  the standard probability measure on the unit sphere.  For all $N\in \N$ one can find a sequence of rational points $\{\nu_i\}_{i=1}^N\subset S^{n-1}$ and positive weights $\{\lambda_i\}_{i=1}^N$ so that $\mu_N=\sum_{i}\lambda_i\delta_{\nu_i}$ defines a probability measure that converges to $d\nu$ as $N\to\infty$.
		
		Let $\mathcal S$ denote the space of real symmetric $2$-tensors $S$ on $\mathbb R^n$ with $|S|=1$. Each $S\in \mathcal S$ defines a continuous function $T_S(\nu)=\tr_{\nu^{\bot}}S$ on $S^{n-1}$ whose spherical average is 
		$$\int_{S^{n-1}}\tr_{\nu^{\bot}}S\,d\nu = \frac{n-1}{n} \tr\, S.$$
		The family $\{T_S\}_{S\in\mathcal S}$ is equicontinuous and so, given $\varepsilon>0$, we have for all  $N$ large enough that
		$$\sup_{S\in\mathcal S}\left|\int_{S^{n-1}}T_S(\nu)d\mu_N-\int_{S^{n-1}}T_S(\nu)d\nu\right|\leq \varepsilon.$$
		This proves the first inequality. The second inequality follows in the same way using  the spherical averages {(see \cite[Theorem~2.4]{bodmann-ehler-graf})\footnote{{Use $|S|_{\nu^{\bot}}^2=|S|^2-2|S(\nu,\cdot)|^2+|S(\nu,\nu)|^2$}}}
		\begin{align*}
			\int_{S^{n-1}}\bigl(\tr_{\nu^{\bot}}S\bigr)^2\,d\nu
			&=\frac1{n(n+2)}\bigl((n^2-3)(\tr\, S)^2+2|S|^2\bigr),\\
			\int_{S^{n-1}}{|S|_{\nu^{\bot}}}^2\,d\nu
			&=\frac1{n(n+2)}\bigl((\tr\, S)^2+(n^2-2)|S|^2\bigr).\\
		\end{align*}
		We now prove the last inequality. Consider $\mathcal S_1$ to be the set of pairs $(S,e)$ where $e\in S^{n-1}$ and $S$ is a complex symmetric $2$-tensor with $|S|=1$ and $S(e,\cdot)=0$.
		For every $(S,e)\in \mathcal S_1$, the function $\nu\mapsto\Psi(e,S)(\nu)=\Psi(e,S,\nu)$ is continuous because, writing  $\nu=\theta e+\sqrt{1-\theta^2}\,\omega$ where $\omega\cdot e=0$  and $\omega\in S^{n-1}$, we have
		\begin{equation}\label{integration.grassmanian}
			\Psi(e,S,\nu)=\frac14\theta^2(1-\theta^2)|S(\omega,\omega)|^2\leq \frac14(1-(e\cdot\nu)^2).
		\end{equation}
		The estimate above also shows that the map $(e,S,\nu)\mapsto\Psi(e,S,\nu)$ is continuous on the compact set $\mathcal S_1\times S^{n-1}$, and thus the family $\{\Psi(e,S)\}_{(e,S)\in\mathcal S_1}$ is equicontinuous. Moreover, the following spherical average holds (see Lemma \ref{psi.average.identity}).
		$$
		\int_{S^{n-1}}\Psi(e,S)(\nu)\,d\nu=\frac{|\tr S|^2+2|S|^2}{4n(n+1)(n+2)}.
		$$
		Reasoning as in the previous cases proves the lemma.
	\end{proof}

	\subsection{Proof of Theorem \ref{codimension.one.thm}}
	\begin{proof}
		In what follows we assume $|h|\leq 1$, use  $C_n$ to denote a generic constant that depends only on $n$ {and $g_0$. We continue to use $O(q)$ to denote a quantity bounded by $C_n|q|$.} Set
		$$
		K=\widehat h(0)\quad\text{and}\quad T=\tr h.
		$$
		With respect to $g_0$-orthonormal coordinates we have
		$
		dV_g=\sqrt{\det(I+h)}\,dV.
		$
		Using
		$$
		\sqrt{\det(I+h)} =1+\frac12T+\frac18T^2-\frac14|h|^2+O(|h|^3)
		$$
		{and $\vol_g(T^n)=1$ we get
			$$
			0=\frac12\int_{T^n}T\,dV+\frac18\int_{T^n}T^2dV-\frac14\int_{T^n}|h|^2dV+O(||h||_{L^{\infty}}||h||_{L^2}^2)
			$$}
		and so
		\begin{equation}\label{eq:volume-constraint}
			\text {tr}K=\frac12\int_{T^n}|h|^2\,dV-\frac14\int_{T^n}T^2\,dV+O(||h||_{L^{\infty}}||h||_{L^2}^2).
		\end{equation}
		{The next lemma uses  ${\bf M}(T_P)\leq{\bf M}_g(T_P+x)$ to show that the $L^2$-norms of $|K|$, $|T|$, and $|\nabla T|$ are much smaller than $||h||_{L^2}$.}
		\begin{lemm}\label{first.estimates}
			We have for some constant $C_n^*$ depending only on $n$ that
			$$|K|\leq C^*_n||h||^2_{L^2},\; ||T||_{L^2}^2\leq C^*_n||h||_{\mathcal W}||h||^2_{L^2},\;\text{and}\; ||\nabla T||_{L^2}^2\leq C^*_n||h||_{\mathcal W^2}||h||^2_{L^2}.$$
		\end{lemm}
		\begin{proof}
			From \eqref{first.inequality}  we have that for every rational hyperplane $P$ and $x\in T^n$
			$$
			R_P(h)(x)\geq-C_nR_P(|h|^2)(x)
			$$
			Integrating over $T^n$ {gives
				$$\tr _PK=\int_{T^n}\tr_Ph\,dV=\int_{T^n}R_P(h)\,dV\geq -C_n\int_{T^n}R_P(|h|^2)\,dV=-C_n||h||^2_{L^2}.$$
			} Hence if $\nu$ is a unit normal vector to $P$ we obtain from \eqref{eq:volume-constraint} that
			$$K(\nu,\nu)=\tr K- \tr _PK\leq C_n||h||^2_{L^2}.$$
			Denseness of rational planes implies the estimate above holds for every unit vector $\nu$. This bounds the largest eigenvalue of $K$ by $C_n||h||^2_{L^2}$. Since $\tr K=O(||h||^2_{L^2})$ we have that
			$$|K|\leq C_n||h||^2_{L^2}.$$ 
			Put $b=h-K$. We have 
			$$R_P(b)=R_P(h)-\tr _PK\geq -C_nR_P(|h|^2)-\tr _PK$$ 
			and so its negative part $(R_P(b))_-$ satisfies
			$$(R_P(b))_-\leq C_nR_P(|h|^2)+C_n||h||^2_{L^2}.$$
			Since $R_P(b)$ has zero average  and the average of $R_P(|h|^2)$ is $||h||^2_{L^2}$ we obtain
			$$
			||R_P(b)||_{L^1(T^n)}=2\int_{T^n}(R_P(b))_-\,dV\leq C_n||h||^2_{L^2}.
			$$
			Fix $\xi\in\mathbb Z^n\setminus\{0\}$ and set $P_\xi=\xi^{\bot}$. The divergence-free condition gives $\widehat h(\xi)(\xi,\cdot)=0$, and therefore, using \eqref{fourier.identity2}, 
			$$
			\widehat T(\xi)=\tr _{P_\xi}\widehat h(\xi)+\widehat h(\xi)(\xi/|\xi|,\xi/|\xi|)=\tr _{P_\xi}\widehat h(\xi){=\widehat{R_{P_\xi}(h)}(\xi)=\widehat{R_{P_\xi}(b)}(\xi).}
			$$
			Consequently, using that $\widehat T(0)=\tr K$, we obtain
			$$
			\sup_{\xi\in\mathbb Z^n}|\widehat T(\xi)|\leq C_n||h||^2_{L^2}.
			$$
			Therefore
			$$
			||T||_{L^2}^2\leq\left(\sup_\xi|\widehat T(\xi)|\right)||T||_{\mathcal W}\leq C_n||T||_{\mathcal W} ||h||^2_{L^2}\leq C_n||h||_{\mathcal W} ||h||^2_{L^2} .
			$$
			Likewise,
			$$
			||\nabla T||_{L^2}^2=\sum_{\xi\in\Z^n}4\pi^2|\xi|^2|\widehat T(\xi)|^2\leq\left(\sup_\xi|\widehat T(\xi)|\right)||T||_{\mathcal W^2}\leq C_n||h||_{\mathcal W^2}||h||_{L^2}^2.
			$$
		\end{proof}
		We now choose the rational directions {to be} used in the proof.  Consider $C^*_n$ given by Lemma \ref{first.estimates}. With $\gamma_n$ a small constant to be chosen later (but depending only on $n$ {and $g_0$}) choose $0<\varepsilon_n\leq \gamma_n$ so that $2C^*_n\varepsilon_n\leq \gamma_n$. With that $\varepsilon_n$ chosen,  consider the rational directions $\{\nu_j\}_{j=1}^N$ with weights  $\{\lambda_j\}_{j=1}^N$ given by Lemma \ref{finite.rational.average}. We consider the rational hyperplanes $P_j=\nu_j^\perp$ and denote for simplicity $T_j=T_{P_j}.$ 
		
		After these choices are made,  any given constant $C$ that depends {on $n$, $g_0$,} and on  the geometry of the finite  collection of tori $\{T_j\}_{j=1}^N$  will also be denoted by $C_n$.
		
		The hypothesis in the theorem implies that  the stable norms for $g$ and $g_0$ agree on integral classes and, by homogeneity and continuity, on $H_{n-1}(T^n,\R)$. The stable norm for $g$ is therefore differentiable at $[T_{j}]\in H_{n-1}(T^n,\Z)$ and so Theorem 5.1 of \cite{cgn} implies that $T^n$ is foliated by area-minimizing hypersurfaces homologous to $T_j$.  Consider   $\{T_{j}(s)\}_{s\in S^1}$ to be a foliation of  $T^n$ by totally geodesic tori homologous to $T_j$. Using the implicit function theorem we find $\delta^*_n$ so that if $|h|_{C^3}<\delta^*_n$  we have for all $j=1,\ldots,N$ and $s\in S^1$, a function  $v_{j,s}\in C^{\infty}(T_{j}(s))$ satisfying
		$$
		\Sigma_{j}(s)=\{x+v_{j,s}(x)\nu_j:x\in T_{j}(s)\}\quad\text{and}\quad \int_{T_{j}(s)}v_{j,s}dA=0
		$$
		so that the hypersurfaces $\Sigma_j(s)$ have constant mean curvature and depend smoothly on $s$. Using the fact that $T^n$ is foliated by area-minimizing hypersurfaces homologous to $T_j$, the maximum principle implies that each $\Sigma_j(s)$ is necessarily an area-minimizing hypersurface. Using the fact that $\sup_{j,s}\|v_{j,s}\|_{C^2(T_{j}(s))}$ can be made arbitrarily small and that  $v_{j,s}$ satisfies a quasi-linear elliptic equation, we improve the estimate to 
		$$\sup_{j,s}|v_{j,s}|_{C^2(T_{j}(s))}\leq C_n|h|_{C^3}.$$
		Thus, with  $a_j=h(\nu_j,\nu_j)$, the linearization of the minimal surface  equation gives
		\begin{equation}\label{EL:equation}
			\Delta_{P_j}v_{j,s}=\frac12\Pi_j^0\partial_{\nu_j}(T+a_j)+O_n(|h|^2_{C^3}),
		\end{equation}
		where  $\Pi_j^0=\Pi^0_{P_j}$.
		Using $h^\top$ to denote {$h|_{P_j}$,} the Taylor expansion for the area of $\Sigma_{j}(s)$ up to cubic order in Section \ref{area.expansion.appendix} gives
		\begin{equation}\label{cubic.order}
			\begin{aligned}
				\text{area}_g(\Sigma_j(s))={\bf M}(T_j)&+\frac12\int_{T_{j}(s)}\tr _{P_j} h\,dA+\frac12\int_{T_{j}(s)}v_{j,s}\partial_{\nu_j}\tr _{P_j} h\,dA\\
				&+\int_{T_{j}(s)} h(\nabla v_{j,s},\nu_j)\,dA+\frac12\int_{T_{j}(s)}|\nabla v_{j,s}|^2\,dA\\
				&+\int_{T_{j}(s)}\left[\frac18(\tr _{P_j} h)^2-\frac14|h^\top|^2\right]dA+O(|h|_{C^3}^3).
			\end{aligned}
		\end{equation}
		Using the fact that $\delta h=0$  and  $\tr _{P_j} h=T-a_j$ we obtain
		\begin{multline*}
			\frac12\int_{T_{j}(s)}v_{j,s}\partial_{\nu_j}\tr _{P_j} h\,dA+\int_{T_{j}(s)} h(\nabla v_{j,s},\nu_j)\,dA
			\\
			=\frac12\int_{T_{j}(s)}v_{j,s}\partial_{\nu_j}(T-a_j)\,dA+\int_{T_{j}(s)} v_{j,s}\partial_{\nu_j}a_j\,dA
			\\
			=\frac12\int_{T_{j}(s)}v_{j,s}\partial_{\nu_j}(T+a_j)\,dA.
		\end{multline*}
		Thus, using \eqref{EL:equation} we conclude
		\begin{multline*}
			\frac12\int_{T_{j}(s)}v_{j,s}\partial_{\nu_j}\tr _{P_j} h\,dA+\int_{T_{j}(s)} h(\nabla v_{j,s},\nu_j)\,dA\\
			=-\int_{T_{j}(s)}|\nabla v_{j,s}|^2\,dA+O(|h|_{C^3}^3).
		\end{multline*}
		Inserting this into \eqref{cubic.order} and using the hypothesis in the theorem we obtain
		$$
		\begin{aligned}
			0={}&\frac12\int_{T_{j}(s)}\tr _{P_j} h\,dA-\frac12\int_{T_{j}(s)}|\nabla v_{j,s}|^2\,dA\\
			&+\int_{T_{j}(s)}\left[\frac18(\tr _{P_j} h)^2-\frac14|h^\top|^2\right]dA+O(|h|_{C^3}^3).
		\end{aligned}
		$$
		Let $v_j$ be the function on $T^n$ whose restriction to $T_j(s)$ is $v_{j,s}$, and set
		$$
		U_j=\int_{T^n}|\nabla_{P_j}v_j|^2\,dV.
		$$
		Integrating the previous identity with respect to the parameter $s$ we obtain
		\begin{equation}\label{U.identity}
			\frac12 U_j=\frac12\int_{T^n}\tr _{P_j} h\,dV+\int_{T^n}\left[\frac18(\tr _{P_j} h)^2-\frac14|h^\top|^2\right]dV+O(|h|_{C^3}^3).
		\end{equation}
		Using Lemma \ref{finite.rational.average} applied to $K= \int_{T^n}hdV$ and  Lemma \ref{first.estimates} we obtain
		$$
		\begin{aligned}
			\frac12\sum_j\lambda_j\int_{T^n}\tr _{P_j} h\,dV&=\frac12\sum_j\lambda_j \tr_{P_j}K\geq\frac{n-1}{2n}\tr K-\frac{\varepsilon_n}{2} |K|\\
			&\geq  \frac{n-1}{2n}\tr K-\frac{\varepsilon_n}{2}C^*_n ||h||^2_{L^2}\!\!\geq\frac{n-1}{2n}\tr K-\gamma_n ||h||^2_{L^2}.
		\end{aligned}
		$$
		Combining with \eqref{eq:volume-constraint} and Lemma \ref{first.estimates}  we see that
		$$
		\frac12\sum_j\lambda_j\int_{T^n}\tr _{P_j} h\,dV\geq \left(\frac{n-1}{4n}-\gamma_n\right)||h||^2_{L^2}+O(||h||_{\mathcal W}||h||_{L^2}^2).
		$$
		Using Lemma \ref{finite.rational.average} applied to $h$ and recalling Lemma \ref{first.estimates} again we obtain 
		\begin{multline*}
			\sum_j\int_{T^n}\lambda_j\left[\frac18(\tr _{P_j} h)^2-\frac14|h^\top|^2\right]dV\\
			\geq \frac{n^2-5}{8n(n+2)}\int_{T^n}T^2dV-\left(\frac{n^2-3}{4n(n+2)}+\gamma_n\right)||h||^2_{L^2}\\
			\geq -\left(\frac{n^2-3}{4n(n+2)}+\gamma_n\right)||h||^2_{L^2}+O(||h||_{\mathcal W}||h||^2_{L^2}).
		\end{multline*}
		
		Inserting these estimates into \eqref{U.identity} we obtain
		\begin{equation}\label{contradiction.estimate}
			\sum_j\lambda_j U_j\geq  \left(\frac{n+1}{2n(n+2)}-4\gamma_n\right)||h||^2_{L^2}+O(||h||_{\mathcal W}||h||^2_{L^2})+O(|h|_{C^3}^3).
		\end{equation}
		We estimate $U_j$ and then we use the  inequality above to conclude that $h=0$. Consider $f_j$ given by Lemma \ref{fourier.lemma} with $P=P_j$. Set $\phi_j=v_j-f_j$. We have from \eqref{EL:equation} that
		$$\Delta_{P_j}\phi_j=\frac12\Pi_j^0\partial_{\nu_j}T+O_n(|h|^2_{C^3}).
		$$
		The Poincar\'e constants of $\{T_j\}_{j=1}^N$ are bounded from below by some $\delta_n^*$. The function $\phi_j$ has vanishing mean on  $T_j+x$  and so we see from the equation above and integration by parts that
		$$\int_{T_j+x}|\nabla_{P_j}\phi_j|^2dA\leq C_n\int_{T_j+x}|\nabla T|^2dA+O(|h|^4_{C^3}).$$
		From Lemma \ref{first.estimates} we deduce
		\begin{equation}\label{gradient.phi}
			\int_{T^n}|\nabla_{P_j}\phi_j|^2dV=O(||h||_{\mathcal W^2}||h||^2_{L^2}+|h|^4_{C^3}).
		\end{equation}
		Hence we obtain
		\begin{multline}\label{uj.estimate}
			U_j\leq  \frac{6}{5}\int_{T^n}|\nabla_{P_j}f_j|^2\,dV+C_n\int_{T^n}|\nabla_{P_j}\phi_j|^2\,dV\\=\frac{6}{5}\int_{T^n}|\nabla_{P_j}f_j|^2\,dV+O(||h||_{\mathcal W^2}||h||^2_{L^2}+|h|^4_{C^3}).
		\end{multline}
		Using the notation set in Lemma \ref{finite.rational.average} we have from Lemma \ref{fourier.lemma} that
		$$\int_{T^n}|\nabla_{P_j}f_j|^2\,dV=\sum_{\xi \in \Z^n,\, {\xi\neq0}} \Psi(\xi/|\xi|,\widehat h(\xi),\nu_j).$$
		Thus we obtain from Lemma \ref{finite.rational.average} and Lemma \ref{first.estimates} that
		$$
		\begin{aligned}\sum_j\int_{T^n}\lambda_j|\nabla_{P_j}f_j|^2dV&\leq  \sum_{\xi \in \Z^n,\,{\xi\neq0}}\! \frac{|\tr \widehat h(\xi)|^2+2|\widehat h(\xi)|^2}{4n(n+1)(n+2)}\!+\!\gamma_n\sum_{\xi \in \Z^n,\,{\xi\neq0}}\!|\widehat h(\xi)|^2\\
			&\leq \left(\frac{1}{2n(n+1)(n+2)}+\gamma_n\right)||h||^2_{L^2}+C_n||T||^2_{L^2}.
		\end{aligned}
		$$
		Hence, using this estimate in \eqref{uj.estimate} and then Lemma \ref{first.estimates} to estimate  $||T||^2_{L^2}$, we obtain
		$$\sum_j\lambda_j U_j\leq  \frac{6}{5}\left(\frac{1}{2n(n+1)(n+2)}+\gamma_n\right)||h||^2_{L^2}+O(||h||_{\mathcal W^2}||h||^2_{L^2}+|h|^4_{C^3})
		.$$
		Combining with \eqref{contradiction.estimate} we see that
		$$
		\begin{aligned}
			&\frac{6}{5}\left(\frac{1}{2n(n+1)(n+2)}+\gamma_n\right)||h||^2_{L^2}\\
			&\geq  \left(\frac{n+1}{2n(n+2)}-4\gamma_n\right)||h||^2_{L^2}+O(||h||_{\mathcal W^2}||h||^2_{L^2}+|h|^3_{{C^3}}).
		\end{aligned}
		$$
		{We have $ \frac{6}{5}\frac{1}{2n(n+1)(n+2)}<\frac{n+1}{2n(n+2)}$. Thus we can find $\gamma_n$ depending only on $n$ so that the inequality above becomes, for some $\alpha=\alpha(n)>0$,
			\begin{equation}\label{final.identity}
				0\geq \alpha||h||^2_{L^2}+O(||h||_{\mathcal W^2}||h||^2_{L^2}+|h|^3_{{C^3}})
		\end{equation}}
		We have $||h||_{\mathcal W^2}\leq {C_{n}}|h|_{C^{3+n}}$. If $|h|_{C^{n+10}}\leq 1$, interpolation \cite[p.~125]{Nirenberg1959} gives
		$$
		\|h\|_{C^3(T^n)}^3\leq  {C_{n}}\|h\|_{L^2(T^n)}^{\frac{3(n+7)}{n+10+n/2}}={C_{n}}\|h\|_{L^2(T^n)}^{c_n}\|h\|_{L^2(T^n)}^2,\quad\text{where }c_n>0.
		$$
		Hence, we can find $\delta_n$, {depending only on $n$ and $g_0$, so that if  $$||h||_{L^2}+|h|_{C^{n+10}}<\delta_n$$ then the term
			$O(||h||_{\mathcal W^2}||h||^2_{L^2}+|h|^3_{{C^3}})$ in \eqref{final.identity} becomes $\geq - \alpha/2||h||^2_{L^2}.$ This  implies that $h=0$.}
	\end{proof}

	\section{Non-rigidity}\label{non-rigid}
	
	{Fix a Euclidean unit-volume metric $g_0$ on  $T^n=\R^n/\Z^n$. Geometric quantities are computed with respect to $g_0$ unless it is mentioned otherwise.  Set $a=|dx_n|$ and consider $e_n$ to be the unit vector so that $e_n^*=dx_n/a$.

		Given a metric $g$ on $T^n$, consider the symmetric $2$-tensor $G=\det(g)g^{-1}$ (with respect to a $g_0$-orthonormal frame). The stable norm of $g$ can be computed explicitly for a large class of metrics using the following remarks.
		
		Consider a parallel form $v\in H^1(T^n,\Z)$ and $\Sigma$ an integral cycle homologous to $PD(v)$  with a unit normal vector $\nu$ defined almost everywhere. {Here $PD$ denotes the Poincar\'e dual.} Then $\int_{\Sigma}\nu\, dA=v$ meaning that for every vector $e\in\R^n$
		$$\int_{\Sigma}\nu\cdot e\, dA=\int_{\Sigma}\iota_e dV=\int_{T^n}v\wedge \iota_e dV= \int_{T^n}v(e) dV=v(e).$$
		The slice $\Sigma_t=\langle\Sigma,x_n,t\rangle$, which informally corresponds to $\Sigma\cap\{x_n=t\}$ (see \cite[Section 6.4]{leon}), is well defined for almost all $t\in \R/\Z$ and has a  unit normal vector denoted by $\nu_t$. With $\iota:\{x_n=t\}\rightarrow T^n$ being the inclusion, we have $[\Sigma_t]=PD(\iota^*v)$ and so, for every  $e\in\R^{n-1}$,
		$$\int_{\Sigma_t}\nu_t\cdot e\, dA=v(e)\area(\{x_n=t\})=av(e)\quad\text{and}\quad\int_{\Sigma_t}\nu_t\cdot e_n\, dA=0$$
		Thus we see that $\int_{\Sigma_t}\nu_t\, dA=av_0$, where $v=v_0+v_ne_n^*$ and $v_0\bot e_n^*$. 
		
		The next theorem implies Theorem \ref{lots.deformations}.
		\begin{thm}\label{stable.warp} Assume the  smooth metric $g$ is such that $G=g(x_n)+dx_n^2$, where $x_n\mapsto g(x_n)$ is  a smooth family of flat metrics on $T^{n-1}$.
			
			We have for all $v=v_0+v_ne_n^*\in H^1(T^n,\R)$, where $v_0\bot e_n^*$, that
			$$
			||PD(v)||_{s,g}=\sqrt{\left(\int_{\R/\Z}|v_0|_{g(t)}dt\right)^{2}+a^2v_n^2}.
			$$
			Furthermore, the following properties hold:
			\begin{itemize}
				\item[(i)] Every primitive class in $H_{n-1}(T^n,\Z)$ admits a foliation by minimal tori;
				\item[(ii)] The unit stable ball $\{w\in H_{n-1}(T^n,\R):||w||_{s,g}\leq 1\}$ is $C^2$ if and only if it is an ellipsoid;
				\item[(iii)] If $n\geq 3$ there are examples where $g$ is not flat and the stable norm is not $C^2$.
			\end{itemize}
	\end{thm}}
	\begin{proof}
		Consider a primitive $v\in H^1(T^n,\Z)$ with $v_0\neq 0$, $\Sigma$ an integral cycle homologous to $PD(v)$ with a unit normal vector $\nu$, and  the slice $\Sigma_t=\langle\Sigma,x_n,t\rangle$ with a unit normal vector $\nu_t$. Consider as well the pointwise orthogonal decomposition $\nu=\nu_0+\nu_ne_n$. We have $\nu_t=\nu_0/|\nu_0|$ and $|\nabla^\Sigma x_n|=a|\nu_0|$.
		
		Using $\int_{\Sigma}\nu\, dA=v$  we see that
		\begin{multline}\label{area.minimizing.stable}
			{\bf M}_{g}(\Sigma)=\int_\Sigma \sqrt{G(\nu,\nu)}dA=\int_\Sigma \sqrt{|\nu_0|^2_{g(x_n)}+a^2\nu_n^2}dA\\
			\geq\left(\left(\int_{\Sigma}|\nu_0|_{g(x_n)}dA\right)^2+a^2\left(\int_{\Sigma}\nu_ndA\right)^2\right)^{1/2}\\=\left(\left(\int_{\Sigma}|\nu_0|_{g(x_n)}dA\right)^2+a^2v_n^2\right)^{1/2}.
		\end{multline}
		We also have for almost all $t$ that $\int_{\Sigma_t}\nu_t\, dA=av_0$ and so
		$$a|v_0|_{g(t)}=\left| \int_{\Sigma_t}\nu_t dA\right|_{g(t)}\leq\int_{\Sigma_t}|\nu_t|_{g(t)}dA=a\int_{\Sigma_t}{|\nu_0|_{g(t)}}{|\nabla^\Sigma x_n|}^{-1}dA.$$
		Integrating in $t$ and using the co-area formula, we obtain
		$$\int_{\R/\Z}|v_0|_{g(t)}dt\leq \int_{\R/\Z}\int_{\Sigma_t}{|\nu_0|_{g(t)}}{|\nabla^\Sigma x_n|}^{-1}dAdt=\int_{\Sigma}|\nu_0|_{g(x_n)}dA.$$
		Thus 
		\begin{equation}\label{area.minimizing.stable0}
			{\bf M}_{g}(\Sigma)\geq \sqrt{\left(\int_{\R/\Z}|v_0|_{g(t)}dt\right)^{2}+a^2v_n^2}.
		\end{equation}
		Minimizing the left-hand side among all  cycles in $[PD(v)]$ gives
		$$||PD(v)||_{s,g}\geq \sqrt{\left(\int_{\R/\Z}|v_0|_{g(t)}dt\right)^{2}+a^2v_n^2}.$$
		We now show the reverse inequality. Set $A=\int_{\R/\Z}|v_0|_{g(t)}dt$ and consider
		$$\lambda=v_0+\frac{v_n}{A}|v_0|_{g(x_n)}e_n^*.$$
		The $1$-form $\lambda$ is closed  and cohomologous to $v$, since $e_n^*=a^{-1}dx_n$ and the coefficient of $dx_n$ in $\lambda-v$ has zero average. Thus there is a smooth map $$f:T^n\rightarrow \R/\Z\quad\text{with}\quad df=\lambda.$$
		We have $|d f|>0$ and so $\{f^{-1}(t)\}_{t\in \R/\Z}$ gives a foliation of $PD(v)$ by smooth embedded tori.  Using the co-area formula we obtain
		$$\int_{\R/\Z}\text{area}_g(f^{-1}(t))dt=\int_{T^n}|df|_{g}dV_g.$$
		We have $|df|_{g}^2=g^{-1}(df,df)$ and $dV_g=\sqrt{\det(g)}dV$. Thus
		\begin{multline*}
			\int_{\R/\Z}\text{area}_g(f^{-1}(t))dt=\int_{T^n}\sqrt{g^{-1}(df,df)}\sqrt{\det(g)}dV\\
			=\int_{T^n}\sqrt{G(df,df)}dV=\frac{1}{A}\int_{T^n}\sqrt{|v_0|_{g(x_n)}^2A^2+a^2v_n^2|v_0|_{g(x_n)}^2}dV\\
			=\frac{1}{A}\int_{\R/\Z}|v_0|_{g(t)}\sqrt{A^2+a^2v_n^2}dt=\sqrt{\left(\int_{\R/\Z}|v_0|_{g(t)}dt\right)^{2}+a^2v_n^2}.
		\end{multline*}
		Combining with \eqref{area.minimizing.stable0} we see that each $f^{-1}(t)$ is area-minimizing with area
		$$||PD(v)||_{s,g}=\sqrt{\left(\int_{\R/\Z}|v_0|_{g(t)}dt\right)^{2}+a^2v_n^2}.$$
		Scaling implies that the identity above  holds for all $v\in H^1(T^n,\Q)$, $v_0\neq 0$ and continuity implies that it holds for all $v\in H^1(T^n,\R)$.
		
		If $v\in H^1(T^n,\Z)$ is primitive and such that $v_0\neq 0$, we showed that $PD(v)$ is foliated by area-minimizing tori. If $v=dx_n$, then each $\{x_n=t\}$ achieves equality in \eqref{area.minimizing.stable} and thus provides a foliation by area-minimizing tori. This proves (i). 
		
		We now prove (ii). Consider the norm on $H^1(T^{n-1},\R)$ defined as
		\begin{equation*}\label{norm.B}
			B(v_0)=\int_{\R/\Z}|v_0|_{g(t)}dt
		\end{equation*}
		and the function
		$$\Phi:H^1(T^{n-1},\R)\rightarrow \R, \quad \Phi(u)=||PD(u+ e_n^*)||_{s,g}=\sqrt{B(u)^2+a^2}.$$
		If the unit stable ball is $C^2$ then $\Phi$ is $C^2$ at the origin. We have 
		$D^2\Phi_{|0}(u,u)=B^2(u)/a$ and so the norm $B$ comes from an inner product on $H^1(T^{n-1},\R)$. Thus $v\mapsto ||PD(v)||^2_{s,g}$ is a quadratic polynomial on $v$ and so 
		$$\{v\in H^1(T^n,\R):||PD(v)||_{s,g}\leq 1\}$$
		is an ellipsoid.  The map $PD:H^1(T^n,\R)\rightarrow H_{n-1}(T^n,\R)$ is linear and invertible. Thus the unit stable ball in $H_{n-1}(T^n,\R)$ is an ellipsoid. The other direction is immediate.
		
		We prove (iii). Extend $e_n^*$ to an orthonormal basis $\{e_j^*\}_{j=1}^n$. Choose a positive non-constant smooth function $\phi$ on $\R/\Z$. Consider
		$$G=\phi^2(x_n)((e_1^*)^2+\ldots+(e_{n-2}^*)^2)+(e_{n-1}^*)^2+dx_n^2=g(x_n)+dx_n^2.$$
		Set $g$ to be the unique metric so that $G=\det(g)g^{-1}$. We have 
		$$
		g=a^{\frac{2}{n-1}}\phi^{-\frac{2}{n-1}}((e_1^*)^2+\ldots+(e_{n-2}^*)^2)+a^{\frac{2}{n-1}}\phi^{\frac{2(n-2)}{n-1}}((e_{n-1}^*)^2+a^{-2}(e_n^*)^2)
		$$
		Writing $g=h(x_n)+\phi^{\frac{2(n-2)}{n-1}}a^{-\frac{2(n-2)}{n-1}}(e_n^*)^2$, we see that the induced metric on the minimal tori $\{x_{n}=t\}$ is flat, but they are not totally geodesic if $h'(x_n)\neq 0$. Thus, from Gauss equation, $g$ is not flat.
		We are left to address the regularity of the stable norm. Consider the norm on $\R^2$  given by
		$$N(x,y)=||PD(xe_1^*+y e_{n-1}^*)||_{s,g}=\int_{\R/\Z}\sqrt{x^2\phi^2(t)+y^2}dt.$$
		The norm $N$ on $\R^2$ does not satisfy the parallelogram law and so $N$ does not come from an inner product. Thus the unit stable ball is not an ellipsoid and so, by (ii),  the stable norm is not $C^2$.
		
	\end{proof}

	When $n=3$ we obtain metrics which preserve the stable norm of Euclidean metrics. This is Theorem \ref{bangert.question}.
	{\begin{thm}\label{n=3.example} Consider a unit-volume flat metric $g_0$ on $T^3$. There is a smooth family of metrics $\{{g}(s)\}_{0\leq s<+\infty}$ so that
			\begin{itemize}
				\item ${g}(0)=g_0$ and  ${g}(s)$ is  not flat for $s>0$;
				\item the stable norm on $H_2(T^3,\R)$ remains fixed;
				\item  $\vol_{{g}(s)}(T^3)<1$ if $s\neq 0$ and $\lim_{s\to +\infty}\vol_{{g}(s)}(T^3)=0.$
			\end{itemize} 
	\end{thm}}
	
	\begin{proof}
		Extend $e_3^*$ to a $g_0$-orthonormal basis $\{e_j^*\}_{j=1}^3$. Fix $b\geq 1$, and set
		$$c=\int_{\R/\Z}\sqrt{b^2\cos^2(2\pi t)+b^{-2}\sin^2(2\pi t)}dt\geq 1.$$
		Consider $G=a^2h(x_3)+dx_3^2$ where, with respect to the coframe $\{e_1^*,e_2^*\}$, 
		$$
		h(x_3)=c^{-2}
		\begin{pmatrix}
			\cos\theta&\sin\theta\\
			-\sin\theta&\cos\theta
		\end{pmatrix}
		\begin{pmatrix}
			b^2&0\\
			0&b^{-2}
		\end{pmatrix}
		\begin{pmatrix}
			\cos\theta&-\sin\theta\\
			\sin\theta&\cos\theta
		\end{pmatrix},\quad \theta=2\pi x_3.
		$$		
		Writing $v_0=|v_0|(\cos\alpha e_1^*+\sin\alpha e_2^*)$, we have
		\begin{multline*}
			\int_{\R/\Z}|v_0|_{h(t)}dt=\frac{|v_0|}{c}\int_{\R/\Z}\big(b^2\cos^2(2\pi t+\alpha)+b^{-2}\sin^2(2\pi t+\alpha)\big)^{1/2}dt
			=|v_0|.
		\end{multline*}
		Thus, if $\widetilde g$ denotes the unique metric so that $G=\det(\widetilde g)\widetilde g^{-1}$, we have from Theorem \ref{stable.warp} that $||PD(v)||_{s,\widetilde g}=a|v|$ for all $v\in H^1(T^3,\R)$. Set $g=a^{-1}\widetilde g$. Since areas scale by $a^{-1}$ when $\widetilde g$ is replaced by $g$, we have
		$$||PD(v)||_{s,g}=a^{-1}||PD(v)||_{s,\widetilde g}=|v|=||PD(v)||_{s,g_0}.$$
		Since $\det G=a^6c^{-4}$, this metric is
		\begin{multline}\label{metric.g}
			g=b^{-2}\big(\cos(2\pi x_3)e_1^*-\sin(2\pi x_3)e_2^*\big)^2\\
			+b^2\big(\sin(2\pi x_3)e_1^*+\cos(2\pi x_3)e_2^*\big)^2+c^{-2}(e_3^*)^2.
		\end{multline}
		If $b=1$ then $c=1$ and so $g=g_0$. We have $\vol_{g}(T^3)=c^{-1}<1$ if $b>1$ and 
		$$\lim_{b\to\infty}\vol_{g(b)}(T^3)=0.$$
		The metric induced by ${g}$ on $\{x_3=t\}$ is flat and $\{x_3=t\}$ is a minimal torus which is not totally geodesic if $b>1$.  Hence ${g}$ is not flat if $b>1$ by Gauss equation.  The family $\{{g}(s)\}_{s\geq 0}$ comes from considering $b=1+s$ in \eqref{metric.g}.
	\end{proof}

	\section{Volume spectrum rigidity}\label{almgren-pitts-rigid}
	We say that the volume spectra of two metrics $g$ and $g'$ on a closed manifold $M$ {strongly coincide} if for every $k\in\mathbb N$ and every finite covering map $p:M'\to M$,
	$$
	\omega_k(M',p^*g')=\omega_k(M',p^*g).
	$$
	\begin{thm}\label{spectrum.rigidity} Consider  $g_0$ a unit-volume flat metric on $T^n$, $n\geq 2$. There is an open neighborhood $\mathcal U$ of $g_0$ in smooth norm so that if the volume spectrum of $g\in \mathcal U$ strongly coincides with the volume spectrum of $g_0$, then $g$ is isometric to $g_0$.
	\end{thm}
	\begin{proof}
		Let $\mathcal U$ be the open neighborhood of $g_0$ given by Corollary \ref{codimension.one.cor} and let $g\in \mathcal U$. The Weyl law for the volume spectrum \cite[Theorem 1.1]{lmn} and the coincidence of the spectra on $T^n$ give
		$$
		\vol_g(T^n)=\vol_{g_0}(T^n)=1.
		$$
		Consider $P$ a rational $(n-1)$-plane. Choose the primitive covector $\alpha\in(\Z^n)^*$ such that $P=\ker\alpha$, set $\Lambda=P\cap\Z^n$, and choose $q\in\Z^n$ such that $\alpha(q)=1$. For every $m,l\in\N$ consider the sublattices of $\Z^n$
		$$
		\Lambda_{m,l}=m\Lambda\oplus l\Z q\quad\text{and}\quad\Lambda_m=m\Lambda.
		$$
		Set  $T_{m,l}=\R^n/\Lambda_{m,l}$ and $T_m=P/\Lambda_m$. $T_m$ is an embedded $(n-1)$-torus of $T_{m,l}$ and the projection map from $T_{m,l}$ to $T^n$ is a finite cover. Furthermore $T_{1,1}=T^n$. Here and below, $g$ and $g_0$ on $T_{m,l}$ denote their pullbacks under this covering map.
		
		\begin{prop}\label{long.cyclic.width} We have
			$$
			\lim_{m\to\infty}\liminf_{l\to\infty}\frac{\omega_1(T_{m,l},g)}{m^{n-1}}=\lim_{m\to\infty}\limsup_{l\to\infty}\frac{\omega_1(T_{m,l},g)}{m^{n-1}}=2||T_P||_{s,g}.
			$$
		\end{prop}
		\begin{proof}
			Fix $m\in\N$. We first prove the upper bound. 
			
			{From \cite[Section 8]{caffarelli-lalave}} we find a Caccioppoli set $E\subset\R^n$ and $d=d(P,g)$ such that
			$$
			\begin{aligned}
				&E+y=E\quad\text{for every }y\in\Lambda,\quad E+q\subset E,\quad \partial E\subset \alpha^{-1}(-d,d),\\
				&\text{and}\quad{\bf M}_g(\partial E/\Lambda)=||T_P||_{s,g}.
			\end{aligned}
			$$
			Set $\Sigma_m=\partial E/\Lambda_m$ which projects, for all $l\in\N$, to a cycle in $T_{m,l}$ representing $[T_m]$, and
			\begin{equation}\label{plane.like.mass}
				{\bf M}_g(\Sigma_m)=m^{n-1}||T_P||_{s,g}.
			\end{equation}
			We often identify $\Sigma_m$ and its projection to $T_{m,l}$. 
			In what follows, $C$ denotes a generic constant depending on {$T_P$} and $g$ but not on $m$ or $l$. 
			{\begin{lemm}\label{path.gmt}
					There is a continuous map in the flat topology with no concentration of mass
					$$\Gamma_m:[0,1]\rightarrow \mathcal Z_{n-1}(\R^n/\Lambda_m,\Z_2)$$
					so that
					$$\Gamma_m(0)=\Sigma_m,\quad \Gamma_m(1)=\Sigma_m+q, $$
					\begin{equation}\label{zipper.bound}
						\text{and}\quad \sup_t{\bf M}_g(\Gamma_m(t))\leq {\bf M}_g(\Sigma_m)+Cm^{n-2}.
					\end{equation}
			\end{lemm}}
			\begin{proof}
				Consider an $L^1$-continuous path $t\in[0,1]\mapsto \Theta_m(t)$ of Caccioppoli sets in $T_m$ so that
				\begin{equation}\label{zipper.bound2}
					\Theta_m(0)=\emptyset,\quad\Theta_m(1)=T_m,\quad\text{and}\quad\sup_t{\bf M}_g(\partial\,\Theta_m(t))\leq Cm^{n-2}.
				\end{equation}
				To achieve that, construct such a path on  $T_1$ and transport it to $T_m$ via the map $x\mapsto mx$. 
				
				Set $G_m=(E\setminus(E+q))/\Lambda_m$ and let $\pi:\R^n/\Lambda_m\to T_m$ be the Euclidean projection parallel to $q$. Set 
				$$
				W_t=G_m\cap\pi^{-1}(\Theta_m(t)),\quad \Gamma_m(t)=\partial((E/\Lambda_m)\setminus W_t),\quad 0\leq t\leq1.
				$$
				We have $\Gamma_m(0)=\Sigma_m$, $\Gamma_m(1)=\Sigma_m+q$, and the decomposition, as mod $2$ currents,
				$$\Gamma_m(t)=(\Sigma_m+q)\llcorner \pi^{-1}(\Theta_m(t))+\Sigma_m\llcorner \pi^{-1}(T_m\setminus\Theta_m(t))+\partial \pi^{-1}(\Theta_m(t))\llcorner G_m.$$
				Translation by $q$ is a $g$-isometry and $\pi(x+q)=\pi(x)$.  Thus  
				$${\bf M}_g((\Sigma_m+q)\llcorner \pi^{-1}(\Theta_m(t)))={\bf M}_g(\Sigma_m\llcorner \pi^{-1}(\Theta_m(t))).$$ Hence the mass of the first two terms is at most ${\bf M}_g(\Sigma_m)$. For the third term we have that
				$${\bf M}_g(\partial \pi^{-1}(\Theta_m(t))\llcorner G_m)\leq Cm^{n-2}.$$
				Indeed, $G_m$ is contained in $\alpha^{-1}([-(d+1),d+1])/\Lambda_m$ and so  $\partial \pi^{-1}(\Theta_m(t))\llcorner G_m$ is contained in the image of $\partial^*\Theta_m(t)\times[-(d+1),d+1]$ under the map $(x,s)\mapsto x+sq$. This map has uniformly bounded $(n-1)$-Jacobian and so the estimate follows from  \eqref{zipper.bound2}. The map $t\mapsto \Gamma_m(t)$ is continuous in the flat topology and can be made to have no concentration of mass. Thus \eqref{zipper.bound} follows.
			\end{proof}
			{Translate the path given by Lemma \ref{path.gmt}} successively by $q,\ldots,(l-1)q$ and project it to $T_{m,l}$. The concatenation of these paths gives  a  loop  
			$$\tilde \Phi:\R/(l\Z)\rightarrow \mathcal Z_{n-1}(T_{m,l},\Z_2)$$
			made of cycles homologous mod 2 to $T_m$. Consider the loop 
			{$$ \Phi:\R/(l\Z)\rightarrow \mathcal Z_{n-1}(T_{m,l},\Z_2), \quad \Phi(\theta)=\tilde\Phi(\theta)+\Sigma_m$$}
			made of homologically trivial cycles which is a $1$-sweepout without concentration of mass. We have from  \eqref{plane.like.mass} and \eqref{zipper.bound} that
			$$
			\omega_1(T_{m,l},g)\leq\sup_{\theta}{\bf M}_g(\Phi(\theta))\leq 2m^{n-1}||T_P||_{s,g}+Cm^{n-2}.
			$$
			Thus
			$$\lim_{m\to\infty}\limsup_{l\to\infty}\frac{\omega_1(T_{m,l},g)}{m^{n-1}}\leq 2||T_P||_{s,g}.$$
			We now prove the lower bound. Set
			$$I_{m,l}=\inf \{{\bf M}_g(\partial \Omega):\Omega\text{ a Caccioppoli set with }\vol_g(\Omega)=\vol_g(T_{m,l})/2\}.$$
			It follows from Claim 5.2 and Remark 5.4 of \cite{mnindex} that
			$$
			\omega_1(T_{m,l},g)\geq I_{m,l}.
			$$
			Hence to prove the lower bound, it suffices to show that
			\begin{equation}\label{long.isoperimetric}
				\liminf_{l\to\infty}I_{m,l}\geq 2m^{n-1}||T_P||_{s,g}.
			\end{equation}
			Fix $m\in\N$ and choose, for each $l$, a nearly optimal  Caccioppoli set $\Omega_l$ for the constant $I_{m,l}$.  The covector $\alpha\in (\Z^n)^*$ induces a map, still denoted by $\alpha$, from $T_{m,l}$ to $\R/(l\Z)$.  We abuse notation and denote the sliced current $\langle \Omega_l, \alpha,t\rangle$ simply by $\Omega_l\cap \alpha^{-1}(t)$.
			
			{The next lemma says that, provided $l$ is large, $\Omega_l$ will almost contain some torus $\alpha^{-1}(t_1)$ and almost miss some other torus $\alpha^{-1}(t_2)$.}
			\begin{lemm}\label{slicing.argument} Given $\varepsilon>0$, we can find for all $l$ large an interval $I\subset \R/(l\Z)$ with endpoints $t_1,t_2$ so that the sliced currents $\Omega_l\cap \alpha^{-1}(t_i)$ are well defined and 
				$${\bf M}_g(\Omega_l\cap \alpha^{-1}(t_1))\geq {\bf M}_g(\alpha^{-1}(t_1))-\varepsilon\quad\text{and}\quad {\bf M}_g(\Omega_l\cap \alpha^{-1}(t_2))\leq \varepsilon$$
			\end{lemm}
			\begin{proof}
				We use $C_m$ to denote a generic constant allowed to depend on $T_P$, $g$, $m$, but not on $l$. 
				
				The upper bound that was proven gives $I_{m,l}\leq C_m$ and so
				$
				{\bf M}_g(\partial\Omega_l)\leq C_m.
				$
				From Lemma 4.5 in \cite{leon} we have that the slicing of $\Omega_l$ and its complement $\Omega_l^c$ is well defined for almost all $t$ and
				\begin{equation}\label{coarea}
					\int_{\R/(l\Z)}{\bf M}_g\bigl(\partial(\Omega_l\cap\alpha^{-1}(t))\bigr)\,dt\leq C_m{\bf M}_g(\partial\Omega_l)\leq C_m.
				\end{equation}
				Let $I_1$ and $I_2$  be the subsets of those $t\in \R/(l\Z)$ so that
				$${\bf M}_g( \Omega_l^c\cap\alpha^{-1}(t))>\varepsilon\quad\text{and}\quad {\bf M}_g(\Omega_l\cap\alpha^{-1}(t))>\varepsilon,$$
				respectively.
				We want to show that neither $I_1$ nor $I_2$  has full measure in $\R/(l\Z)$.  Suppose that $I_1$ has full measure and consider $J=I_1\cap I_2$.
				
				There is a uniform bound (depending on $g$, $T_P$, and $m$) on the isoperimetric constant of all $\alpha^{-1}(t)$. Thus if  $t\in J$ we have
				{\begin{align*}
						{\bf M}_g\bigl(\partial(\Omega_l\cap\alpha^{-1}(t))\bigr)& \geq C_m\min\{{\bf M}_g(\Omega_l^c\cap\alpha^{-1}(t)),\,{\bf M}_g(\Omega_l\cap\alpha^{-1}(t))
						\}^{\frac{n-2}{n-1}}\\
						&\geq C_m\varepsilon^{\frac{n-2}{n-1}}.
				\end{align*}}
				Hence we obtain from \eqref{coarea} that $|J|\leq C_m\varepsilon^{\frac{2-n}{n-1}}$. Moreover, we have
				$$
				\vol_g(\Omega_l)=\int_{\R/(l\Z)}\int_{\Omega_l\cap\alpha^{-1}(t)}\frac{1}{|\nabla_g\alpha|}\,dA_g\,dt.
				$$
				Thus, using  that $\R/(l\Z)$ equals $J\cup I_1\setminus J$ up to a set of measure zero, we obtain
				$$
				\vol_g(\Omega_l){=} \int_{J}\int_{\Omega_l\cap\alpha^{-1}(t)}\frac{1}{|\nabla_g\alpha|}\,dA_g\,dt+\int_{I_1\setminus J}\int_{\Omega_l\cap\alpha^{-1}(t)}\frac{1}{|\nabla_g\alpha|}\,dA_g\,dt.
				$$
				The first term on the right is bounded by $\vol_g(\alpha^{-1}(J))$ and  thus bounded by $C_m|J|$. The term $|\nabla_g\alpha|$ is bounded above and below and so
				$$\int_{I_1\setminus J}\int_{\Omega_l\cap\alpha^{-1}(t)}\frac{1}{|\nabla_g\alpha|}\,dA_g\,dt\leq C_m l\varepsilon  $$
				Thus, {using $|J|\leq C_m\varepsilon^{\frac{2-n}{n-1}}$}, we see that
				$$\frac{l}{2}\vol_g(T_{m,1})=\vol_g(\Omega_l)\leq C_m\varepsilon^{\frac{2-n}{n-1}}+C_m l\varepsilon.$$
				Replace $\varepsilon$ by a smaller number if necessary so that $C_m \varepsilon<\frac 14 \vol_g(T_{m,1})$. Hence the inequality above fails to hold for all $l$ sufficiently large. Thus $I_1$ does not have full measure provided  $l$ is suitably large. The same argument applies to $I_2$ provided we replace $\Omega_l$ by $\Omega_l^c$.
			\end{proof}
			Consider $l$ large so that Lemma \ref{slicing.argument} applies. After choosing a suitable orientation on the segment $I\subset \R/(l\Z)$, Lemma 4.5 in \cite{leon}, p.182, gives
			$$
			\partial(\alpha^{-1}(I)\llcorner \Omega_l)=\Omega_l\cap\alpha^{-1}(t_2)-\Omega_l\cap\alpha^{-1}(t_1)+\partial \Omega_l\llcorner \alpha^{-1}(I).
			$$
			Set $E=\alpha^{-1}(t_1)\llcorner \Omega_l^c$.  Then $\alpha^{-1}(t_1)=\Omega_l\cap\alpha^{-1}(t_1)+E$ and so
			$$
			\partial \Omega_l\llcorner \alpha^{-1}(I)+E+\Omega_l\cap\alpha^{-1}(t_2)=\alpha^{-1}(t_1)+\partial(\alpha^{-1}(I)\llcorner \Omega_l).
			$$
			The cycle on the right is homologous to $\Sigma_m$ and its projection to $T^n$ represents $m^{n-1}[T_P]$.  So 
			$${\bf  M}_g(\partial \Omega_l\llcorner \alpha^{-1}(I)+E+\Omega_l\cap\alpha^{-1}(t_2))\geq {\bf M}_g(\Sigma_m).$$
			We have from  Lemma \ref{slicing.argument} that $E+\Omega_l\cap\alpha^{-1}(t_2)$ has mass $\leq 2\varepsilon$ and thus
			$$
			{\bf M}_g(\partial \Omega_l\llcorner \alpha^{-1}(I)) \geq {\bf M}_g(\Sigma_m)-2\varepsilon.
			$$
			We can repeat the argument above replacing $I$ by its complement $I^c$ in $\R/(l\Z)$ and obtain
			$$
			{\bf M}_g(\partial \Omega_l\llcorner \alpha^{-1}(I^c)) \geq {\bf M}_g(\Sigma_m)-2\varepsilon.
			$$
			These two inequalities show that for all $l$ sufficiently large
			$${\bf M}_g(\partial \Omega_l)\geq{\bf M}_g(\partial \Omega_l\llcorner \alpha^{-1}(I^c))+{\bf M}_g(\partial \Omega_l\llcorner \alpha^{-1}(I)) \geq  2{\bf M}_g(\Sigma_m)-4\varepsilon.
			$$
			Therefore
			$$
			\liminf_{l\to\infty}I_{m,l}\geq 2{\bf M}_g(\Sigma_m)-4\varepsilon=2m^{n-1}||T_P||_{s,g}-4\varepsilon.
			$$
			Making $\varepsilon\to 0$ proves \eqref{long.isoperimetric}.
		\end{proof}
		
		The hypothesis gives
		$$
		\omega_1(T_{m,l},g)=\omega_1(T_{m,l},g_0)
		$$
		for every $m,l\in\N$. Applying Proposition \ref{long.cyclic.width} to both metrics gives that for every rational hyperplane $P$ 
		$$
		||T_P||_{s,g}=||T_P||_{s,g_0}.
		$$
		Corollary \ref{codimension.one.cor} shows that $g$ is isometric to $g_0$.
	\end{proof}
	
	\appendix

	\section{{Stable norms and auxiliary formulas}}
	\subsection{Stable norm}\label{stable.norm.appendix}
	{Fix a unit-volume Euclidean metric $g_0$ on $T^n$. Given a closed $k$-form $\alpha$ and $x\in T^n$, the {\em comass} at $x$ is $$|\alpha_x|=\max\{|\alpha_x(v_1,\ldots,v_k)|:v_i\in T_xT^n, |v_i|\leq 1\}$$ and the {\em comass}  is
		$$|\alpha|^*=\max_{x\in T^n}|\alpha_x|.$$
		We obtain a norm on $H^d(T^n,\R)$ as
		$$||\omega||^*=\inf\{|\alpha|^*:\alpha\in [\omega]\}.$$
		The following was proven in \cite[4.10]{Federer74}; see also \cite[4.35]{Gro99}.
		\begin{prop}\label{comass.dual} The comass norm is dual to the stable norm. More precisely, for every $w\in H_d(T^n,\R)$,
			$$||w||_s=\sup\{|\langle w,\omega\rangle|: \omega\in H^d(M,\R), ||\omega||^* \leq 1 \}.$$
		\end{prop}
		The following simple estimate follows from \cite[Proposition 4.3]{HK26}.
		\begin{lemm}\label{forms} Let $E$ be a Euclidean $n$-space
			$\alpha_1,\ldots,\alpha_n\in E^*$. Then, with the comass norm,
			$$|\alpha_1\wedge\cdots\wedge\alpha_n|\leq |\alpha_1|\cdots |\alpha_n|.$$
		\end{lemm}
		
		The next result is well known to experts.
		\begin{prop} Consider a metric $g$  on $T^n$ so that 
			$$||w||_{s,g}=||w||_{s}\quad\text{for all }w\in H_1(T^n,\R).$$
			Then $\vol(T^n,g)\geq 1$ with equality if and only if $g$ is isometric to $g_0$.
		\end{prop}
		\begin{proof}
			Choose a $g_0$-parallel orthonormal coframe $\eta_1,\ldots,\eta_n$ and orient
			$T^n$ so that
			$$\int_{T^n}\eta_1\wedge\cdots\wedge\eta_n=\vol(T^n,g_0)=1.$$
			Set $\omega_i=[\eta_i]\in H^1(T^n,\R)$. From Proposition \ref{comass.dual}  the norm
			$||\cdot||^*$ on $H^1(T^n,\R)$ is dual to the stable norm on
			$H_1(T^n,\R)$. Hence the hypothesis implies that the comass norms induced by
			$g$ and $g_0$ on $H^1(T^n,\R)$ are the same. Since
			$||\omega_i||^*=1$, we have
			$$||\omega_i||^*_{g}=1\quad\text{for every }i.$$
			
			Let $\varepsilon>0$. For each $i$ we can choose a smooth closed one-form with
			$[\alpha_i]=\omega_i$ such that
			$$|\alpha_i|^*_g\leq 1+\varepsilon.$$
			Since the $\alpha_i$ and the $\eta_i$ represent the same cohomology classes,
			$$\int_{T^n}\alpha_1\wedge\cdots\wedge\alpha_n
			=\int_{T^n}\eta_1\wedge\cdots\wedge\eta_n=1.$$
			Using Lemma \ref{forms} pointwise we have
			\begin{align*}
				1&=\left|\int_{T^n}\alpha_1\wedge\cdots\wedge\alpha_n\right|
				\leq \int_{T^n}|\alpha_1\wedge\cdots\wedge\alpha_n|\,d\vol_g\\
				&\leq \int_{T^n}|\alpha_1|\cdots |\alpha_n|\,d\vol_g\leq (1+\varepsilon)^n\vol(T^n,g).
			\end{align*}
			Letting $\varepsilon\to 0$ gives $\vol(T^n,g)\geq 1$.
			If equality holds, then \cite[Proposition~8.5.18]{burago-ivanov-book} implies that $g$ and $g_0$ have the same asymptotic volume. Thus \cite[Theorem~1(b)]{burago-ivanov-gafapaper} implies that $g$ is flat.  Non-isometric flat metrics have distinct marked length spectrum and so $g$ is isometric to $g_0$.
	\end{proof}}

	\subsection{A spherical integral identity}
	{
		
		\begin{lemm}\label{psi.average.identity}
			Let $e\in S^{n-1}$ and let $S$ be a symmetric, possibly complex-valued, $2$-tensor on $\R^n$ satisfying $S(e,\cdot)=0$. Consider $\Psi$ as defined in \eqref{psi.definition}.
			Then
			$$
			\int_{S^{n-1}}\Psi(e,S)(\nu)d\nu=\frac{|\tr S|^2+2|S|^2}{4n(n+1)(n+2)}.
			$$
		\end{lemm}
		
		\begin{proof}
			Denote by $dV_n$ the standard volume form on the unit $n$-sphere with total volume $\omega_n$ and $x_e=\langle x,e\rangle$. We have for $w\in S^{n-1}\cap e^{\bot}=\{x_e=0\}$ $$\Psi(e,S)(\sin te+\cos t w)=\sin^2 t\cos^2t|S(w,w)|^2/4$$ and so
			\begin{multline*}
				\int_{S^{n-1}}\Psi(e,S)(\nu)d\nu\\
				=\frac{1}{\omega_{n-1}}\int_{-\pi/2}^{\pi/2}(\cos t)^{n-2}\int_{\{x_e=0\}}\Psi(e,S)(\sin te+\cos t w)dV_{n-2}(w)dt\\
				=\frac{\omega_{n-2}}{4\omega_{n-1}}\int_{-\pi/2}^{\pi/2}(\cos t)^{n}\sin^2tdt\fint_{\{x_e=0\}}|S(w,w)|^2dV_{n-2}(w).
			\end{multline*}
			Using  \cite[Theorem~2.4]{bodmann-ehler-graf} and $S(e,\cdot)=0$ we see that
			$$\fint_{\{x_e=0\}}|S(w,w)|^2dV_{n-2}(w)=\frac{|\tr S|^2+2|S|^2}{(n-1)(n+1)}.$$
			We also have
			$$\frac{\omega_{n-2}}{4\omega_{n-1}}\int_{-\pi/2}^{\pi/2}(\cos t)^{n}\sin^2tdt=\frac14\fint_{S^{n-1}}x_e^2(1-x_e^2)dV_{n-1}=
			\frac{n-1}{4n(n+2)}.$$
			The result follows from combining the last three identities.
	\end{proof}}
	
	{
		\subsection{Area expansions}\label{area.expansion.appendix} Consider $h$ a symmetric $2$-tensor, $P$ a rational hyperplane with normal unit vector $\nu$, and $f\in C^{\infty}(T_P)$.  We denote by $O(q)$ a quantity whose absolute value is bounded by $C_n|q|$, where $C_n$ depends on $n$, $T_P$, and $g_0$, a flat metric on $T^n$.  Consider
		$$\Sigma=\{x+f(x)\nu:x\in T_P\}\quad\text{where }|f|_{C^2(T_P)}=O(|h|_{C^2}).$$
		Fix an orthonormal basis $\{e_i\}_{i=1}^{n-1}$ of $P$ and set $h^\top=h|_{P}$. We have 
		$$|h^\top|^2=\sum_{i,j=1}^{n-1} h(e_i,e_j)^2.$$
		We show that, with $g=g_0+h$ and assuming $|h|_{C^2}\leq 1$,
		\begin{multline}\label{area.expansion.cubic}
			\text{area}_g(\Sigma)=\text{area}_{g_0}(T_P)+\frac12\int_{T_P}\tr _P h\,dA+\frac12\int_{T_P}f\partial_\nu(\tr _P h)\,dA\\
			+\int_{T_P}h(\nabla f,\nu)\,dA+\frac12\int_{T_P}|\nabla f|^2\,dA
			+\int_{T_P}\left[\frac18\bigl(\tr _P h\bigr)^2-\frac14|h^\top|^2\right]dA\\+\int_{T_P}O(|h|^3_{C^2})\,dA.
		\end{multline}
		Consider $\phi:T_P\rightarrow T^n$, $\phi(x)=x+f(x)\nu$, and $\gamma=\phi^*(g)$. We have
		$$
		\gamma_{ij}=\delta_{ij}+\partial_if\cdot\partial_jf+\left(h(e_i,e_j)+\partial_jfh(e_i,\nu)+\partial_ifh(e_j,\nu)\right)\circ\phi+O(|h|^3_{C^2}).
		$$
		We use the matrix expansion
		$$
		\sqrt{\det(I+A)}=1+\frac12\tr A+\frac18(\tr A)^2-\frac14\tr (A^2)+O(|A|^3).
		$$
		With $A=\gamma-I$, we conclude 
		\begin{multline}\label{formula.phi}
			\text{area}_g(\Sigma)=\text{area}_{g_0}(T_P)+\frac12\int_{T_P}\tr _P h\circ\phi\,dA+\int_{T_P}h(\nabla f,\nu)\circ\phi\,dA\\
			+\frac12\int_{T_P}|df|^2\,dA+\int_{T_P}\left[\frac18\bigl(\tr _P h\circ\phi\bigr)^2-\frac14|h^\top|^2\circ\phi\right]dA+\int_{T_P}O(|h|^3_{C^2})\,dA.
		\end{multline}
		We have
		$$
		\sup_{0\leq s\leq 1}|\nabla^jh|\bigl(x+sf(x)\nu\bigr)=O(|h|_{C^2}),\qquad j=0,1,2,
		$$
		and so, using Taylor's formula, we obtain
		$$
		\tr _P h\circ\phi=\tr _P h+f\partial_\nu(\tr _P h)+O(|h|^3_{C^2}),
		$$
		$$
		h(\nabla f,\nu)\circ\phi=h(\nabla f,\nu)+O(|h|^3_{C^2}),
		$$
		and
		$$
		\frac18\bigl(\tr _P h\circ\phi\bigr)^2-\frac14|h^\top|^2\circ\phi=\frac18\bigl(\tr _P h\bigr)^2-\frac14|h^\top|^2+O(|h|^3_{C^2}).
		$$
		Substituting these three identities into \eqref{formula.phi} we obtain \eqref{area.expansion.cubic}.}
	\bibliographystyle{amsbook}

\end{document}